\documentclass[11pt]{article}
\usepackage[a4paper,hmargin=2.5cm,vmargin=2cm]{geometry}

\usepackage{natbib}%
\usepackage[figuresright]{rotating}

\usepackage{url}
\usepackage{algorithmic}
\usepackage{algorithm}

\usepackage{amsmath,amsthm}
\newtheorem{theorem}{Theorem}

\theoremstyle{definition}

\theoremstyle{remark}

\usepackage[colorlinks,citecolor=blue,urlcolor=blue,bookmarks=false,hypertexnames=true]{hyperref} 
\usepackage{latexsym}
\usepackage{amsmath}
\usepackage{amssymb}
\usepackage{mathtools}
\usepackage{enumitem} 
\usepackage{fancyhdr}
\usepackage{lscape}
\usepackage{booktabs}
\usepackage{array,longtable}
\makeatletter
\def\LT@makecaption#1#2#3{%
  \LT@mcol\LT@cols c{\hbox to\z@{\hss\parbox[t]\LTcapwidth{%
    \captionstyle
    \sbox\@tempboxa{{\floatlegendstyle#1{#2}\floatcounterend}\capstrut #3}%
    \ifdim\wd\@tempboxa>\hsize
      {\floatlegendstyle#1{#2}\floatcounterend}\capstrut #3\par%
    \else
      \hbox to\hsize{\leftlegendglue\box\@tempboxa\hfil}%
    \fi
    \endgraf\vskip\baselineskip}%
  \hss}}}
\LTcapwidth=\dimexpr\textwidth-2\fboxsep\relax
\makeatother
\usepackage{amsfonts}
\usepackage{booktabs}
\usepackage{dsfont} 
\usepackage{multirow}
\usepackage{subfig}
\usepackage{enumitem}
\usepackage{color}
\usepackage{colortbl}
\usepackage{xcolor}
\usepackage{relsize} 
\usepackage{fancyhdr}
\usepackage{braket}
\usepackage{array}
\usepackage{tabularx}
\usepackage[flushleft]{threeparttable} 
\usepackage{caption}
\usepackage{cancel}
\makeatletter
\renewcommand*\env@matrix[1][c]{\hskip -\arraycolsep
  \let\@ifnextchar\new@ifnextchar
  \array{*\c@MaxMatrixCols #1}}
\makeatother
\usepackage{dsfont} 
\usepackage{colortbl}
\usepackage{mathrsfs}
\usepackage{tikz}
\usetikzlibrary{shapes,arrows}
\usepackage{textcomp}
\usepackage[american]{circuitikz}
\usepackage{siunitx}
\usepackage{pgf}
\usepackage{pgfkeys}
\usepackage{xstring}
\usepackage{pdflscape} 
\usepackage{float}
\floatstyle{ruled}
\newfloat{linearmodel}{thp}{lop}[section]
\floatname{linearmodel}{LM}
\makeatletter
\newcommand{\Spvek}[2][r]{%
  \gdef\@VORNE{1}
  \left(\hskip-\arraycolsep%
    \begin{array}{#1}\vekSp@lten{#2}\end{array}%
  \hskip-\arraycolsep\right)}
\def\vekSp@lten#1{\xvekSp@lten#1;vekL@stLine;}
\def\vekL@stLine{vekL@stLine}
\def\xvekSp@lten#1;{\def\temp{#1}%
  \ifx\temp\vekL@stLine
  \else
    \ifnum\@VORNE=1\gdef\@VORNE{0}
    \else\@arraycr\fi%
    #1%
    \expandafter\xvekSp@lten
  \fi}
\makeatother

\title{A Heuristic Algorithm for Traffic Light Synchronization Based on the MAXBAND Model}

\footnotesize\date{May 2018}

\author{
Xavier Cabezas\\
\footnotesize The University of Edinburgh\\
\footnotesize \texttt{J.X.Cabezas@sms.ed.ac.uk}\\
\footnotesize Escuela Superior Polit\'ecnica del Litoral\\
\footnotesize \texttt{joxacabe@espol.edu.ec}\\  
\and
Sergio Garc\'ia\\
\footnotesize The University of Edinburgh\\
\footnotesize \texttt{Sergio.Garcia-Quiles@ed.ac.uk} \\ 
}

\begin{document}

\maketitle

\begin{abstract}
\noindent A widely used approach to solve the synchronization of traffic lights on transport networks is the maximization of the time during which cars start at one end of a street and can go to the other without stopping for a red light (bandwidth maximization). The mixed integer linear model found in the literature, named MAXBAND, can be solved by optimization solvers only for small instances. In this paper we review in detail all the constraints of the original linear model, including those that describe all the cyclic routes in the graph, and we generalize some bounds for integer variables which so far had been presented only for problems that do not consider cycles. Finally, we propose a solution algorithm that uses Tabu Search and Variable Neighbourhood Search and we carry out a computational study to show that it performs very well for large instances.
\end{abstract}


\maketitle

\section{Introduction}\label{sec:Introduction}

Traffic lights have been with us for a long time (since late in the 19th century) and they are used in cities  to control the flow of vehicles. But its use also leads to some problems such as time delays when moving from one place to another and increased pollution due to changes in the speeds of the vehicles. Because of the increase in urban traffic year by year, its timing has been a topic of interest for many researchers. Most of the papers focus on two aspects: minimizing some measure to assess the performance of the traffic (e.g, delays) and maximizing the time that vehicles can drive without stopping for red lights (bandwidth maximization).  

With respect to traffic flow measures, it has been usual to minimize either the overall delay or the number of stops of the vehicles. One of the first models was developed by \cite{Gartner75}. This study demonstrated the feasibility of using mixed integer linear programming to optimize traffic signal settings for practical but small size road networks. In that work a convex nonlinear objective function and linear constraints were considered, but a piecewise linearisation in the objective can be done by increasing the number of constraints in the formulation. The objective function measures the overflow queue wich represents the number of vehicles that are not able to clear an intersection during the preceding green time and the constraints consider the possible loops in which a platoon of vehicles could incur when navigating in a network. However, some realistic constraints were not considered in this case. As an example, different patterns of light changes at street junctions were not taken into account; this would allow to consider either crosswalk times or the synchronization of traffic lights towards cross streets. This method was also used by \cite{ThesisW} to show a similar linear model with some differences. One is that the model is able to decide among different predetermined signal timing plans. Another is that the assumption of a common red-green period width at the signals is relaxed. However, it is not a hard constraint because in that case the least common multiple of all red-green periods could be used as a uniform period, as mentioned in \cite{Kohler}. The linear model was tested on a family of real-world transport networks of up to 146 nodes and 399 arcs. Also \cite{ThesisW} showed a formal proof of the NP-completeness of the signals timing problem. Another reference can be found in \cite{Improta1982}, the authors defined a linear model similar to the one proposed by \cite{Gartner75} and additionally developed a branch and bound procedure with backtracking to solve it which relaxes some assumptions imposed in the original model. Unfortunately, only small examples are reported (up to 9 nodes and 16 arcs).

With regard to bandwidth maximization, one of the first papers was \cite{Little64}. In that paper, the authors presented a geometric and intuitive method for bandwidth maximization on a two-way street with a given fixed and common red-green time period on each signal and preassigned vehicle velocities. Even though there had been some geometric methods developed earlier, this paper presented a systematic algorithm easy to implement. A couple of years later, \cite{Little66} proposed for the first time a mixed integer linear program (MILP) to solve a new version of the problem (more complete) which does not assume a fixed red-green time period, but this is chosen by the model between some given bounds. Upper and lower limits on velocity between adjacent signals and changes in speed are also considered. Solving the model yields a common signal red-green period, velocities between signals and maximal bandwidths on the streets. An extension is provided to solve the same problem on general networks, but only very small instances could be solved. On networks the problem is more difficult because it is necessary to introduce the so-called loop constraints, which permit to model the circular movements that vehicles can do. Some years later, \cite{Little81} introduced a generalization that includes left turns at junctions. This linear model is traditionally called MAXBAND even though this is the name of the code developed to solve it. The first version of MAXBAND could handle problems on networks with only 3 arteries and up to 17 traffic signals. Later \cite{multiband} extended the MAXBAND model by working with variable bandwidth for each street segment (MULTIBAND). This change allowed to incorporate a traffic factor on the objective function. More recently \cite{amband} proposed a new version of MULTIBAND called AM-BAND. This model tries to use better the available green times on both road directions by relaxing a symmetric assumption with respect to the progression line given in the MULTIBAND linear model. In \cite{Xianyu2012} and \cite{linkbased} a variable bandwidths is also considered, their approach uses the criteria of partitioning a large system into smaller subsystems and takes into account the impact of speed variation. However, this method is not based on integer programming.  

The use of heuristics to solve traffic light synchronization problems is common, specially due to the influence of the very successful commercial software for synchronization of traffic signals and traffic management named TRANSYT \citep{transyt}, which is an implementation of a method introduced by \cite{Robertson69}. TRANSYT provides a heuristic solution for a very complete and robust objective function by using microscopic simulation of traffic behaviour and genetic algorithms. Other references of using evolutionary methods can be seen in \cite{genetic3} and \cite{genetic4}. \cite{GartnerNew} proposed a heuristic method for the MAXBAND network problem that in a first stage solves a tree subproblem which considers a measure of interest. Then, the integer variables are fixed to the values obtained and used in a second stage to solve the whole problem. As far as we know there is no other reference about solving the MAXBAND on a complete network.

Since MAXBAND can only be solved by optimization solvers for very small instances, the main aim of this work is to develop a heuristic algorithm for bandwidth maximization. As the algorithm is based on the MAXBAND formulation, first we review the constraints in detail. Later, we generalize some existing bounds based on the ideas outlined in \cite{Little66}. Then, we propose a metaheuristic algorithm to solve the problem on complete networks. We carry out some computational experiments to verify how efficient the method is. The method we propose starts with a feasible solution of the problem and uses basic Tabu Search \citep{GloverTS1} ideas such as the use of a memory structure within an iterative local search process that allows that solutions found during execution to be more diverse. In addition, the search is intensified by a sequence of neighbour solutions search processes, just as in Variable Neighbourhood Search (VNS), see \cite{VNS}. 

One of the key factors that influence the performance of any of the methods mentioned above is the large amount of information they require from the network. Transport networks in the real world are, in some cases, similar to a grid of streets that intersect with each other (grid graph). Therefore, each intersection may have traffic lights for possible vehicle entrances from four different directions. Information of the red and green light times on each of these can be required, as is the case of MAXBAND. Additional data such as the length of the queue of vehicles that wait to cross or turn to another street are also necessary. If real data is not available, a simulation of the network information must be carried out. The parameters to be generated must be consistent to allow feasible solutions; for example, green lights at the same time can not be allowed for all signals that are in the same junction of streets. Unfortunately, we have not been able to access real data and therefore we have opted to simulate them. 

The rest of the chapter is structured as follows. First we provide the main concepts of graph theory that will be used in this chapter as well as an algebraic treatment of cycle basis which will be used in the modelling of TLSP. Then, we review the MAXBAND model in Section \ref{maxband}, where we give a detailed explanation of all its elements and introduce the notation that will be used. Particularly, in Section \ref{intranode} we show how to model the loop constraints with a small number of binary variables. In Section \ref{s4} we generalize the bounds for arterial integer variables that were originally presented by \cite{Little66} for graphs that do not consider loops and extend them for general networks. In Section \ref{TSmethod} we propose a new MILP-based heuristic algorithm that uses Tabu Search and Variable Neighbourhood search. We carry out a computational study to show that our method performs very well for large instances when an initial feasible solution is provided. Some conclusions and ideas for future research are discussed in Section \ref{conclusionsTLSP}.

\section{The MAXBAND Model for a Network}\label{maxband}

Now we define the MAXBAND model notation and give a complete explanation of all variables and constraints involved. Particularly, we discuss in detail how loops are modelled.
 
Let us consider a group of two-way \emph{arteries} (streets) that meet each other in \emph{junctions} forming a transport \emph{network}. This network has some traffic signals in order to regulate traffic. They work with a common \emph{period} that is split in red and green times, this will be used as unit of measurement of time. It is sometimes called \emph{cycle length}, although we will try to avoid the use of this term because it can be confused with the \emph{loops} (also known as cycles) that are present in the network and which play a very important role in the formulation of the problem, as we will see later. Instead, we will use the name \emph{period length}. The distances (time units) that allow to measure the relative location between two signals on the same artery and on different arteries are called \emph{internode offset} and \emph{intranode offset} respectively. A list of offsets for the signals is said to be a \emph{synchronization}.

The MAXBAND model is based on the geometry that we can see in Figure \ref{ggm}. Information is provided for two signals $S_{ai}$ and $S_{aj}$ on an artery $a$. The notation used is essentially the same than in \cite{Little81}. The different variables and parameters are defined as follows:

\newpage

Parameters in Figure \ref{ggm}:

\begin{itemize}

\item $T$: Period length, in seconds.

\item $n_a$: Number of traffic lights (signals) on artery $a$.

\item $r_{ai}$ ($\overline{r}_{ai}$)$:$ Outbound (inbound) red time of signal $i$ on artery $a$, in a fraction of the period.

\item $\tau_{ai}$ ($\overline{\tau}_{ai}$)$:$ An advancement of the outbound (inbound) bandwidth upon leaving $S_i$, in a fraction of the period.

\end{itemize}

Variables in Figure \ref{ggm}:

\begin{itemize}

\item $z$: Signal frequency, in periods per second.

\item $b_a$ ($\overline{b}_a$)$:$ Outbound (inbound) bandwidth on artery $a$, in a fraction of the period.


\item $t_{ij}^a$ ($\overline{t}_{ij}^a$)$:$ Travel time from $S_{ai}$ to $S_{aj}$ in outbound (from~ $S_{aj}$ to $S_{ai}$ in inbound) direction, in periods.

\item $\phi_{ij}^a$ ($\overline{\phi}_{ij}^a$)$:$  Time from the center of red at $S_{ai}$ to the center of red at $S_{aj}$, in periods. The two reds are chosen so that each is immediately to the left (right) of the same outbound (inbound) green band. $\phi_{ij}^a$ ($\overline{\phi}_{ij}^a$) is positive if $S_{aj}$'s center of red lies to the right (left) of $S_{ai}$'s.

\item $w_{ai}$ ($\overline{w}_{ai}$)$:$ Time from the right (left) side of $S_{ai}$'s red to the left (right) side of green band in outbound (inbound) direction, in a fraction of the period.

\item $\Delta_{ai}:$ Time from center of $\overline{r}_{ai}$ to the nearest center of $r_{ai}$, in periods. It is positive from left to right.

\end{itemize}

In Figure \ref{ggm} a two direction (outbound and inbound) street (artery) is depicted in a space-time plane. The continuous lines formed by the union of red and green segments (a period), shown above the signals $S_{ai}$ and $S_{aj}$, represent the periodic behaviour of the traffic lights over time in the outbound direction, while the dash lines correspond to inbound. On each signal both lines do not necessarily overlap. This makes sense since it is possible to consider different movement patterns of vehicles to cross streets from the considered artery and vice versa, as we will see in Section \ref{sarterialc}. If these schemes are not taken into account both lines will coincide, as is the case of long roads where the red lights must be present only to allow the crossing of pedestrians in certain points of the road. Points A, B, C and D will be useful to define the different relationships between the variables and parameters necessary for the formulation.\\

The complete formulation for MAXBAND on a network can be seen in LM \ref{maxbandn}. It has additional parameters and variables (not showed in Figure \ref{ggm}), jointly with the constraints, these are explained next. Similar to \cite{Little81}, in this thesis we do not write $\alpha_{i,i+1}^a$ but $\alpha_i^a$ for any parameter or variable $\alpha$ and their corresponding names with bars for the opposite direction on an artery. 

\begin{landscape}
\begin{figure}[h]
  \centering
    \includegraphics[width=1.5\textwidth]{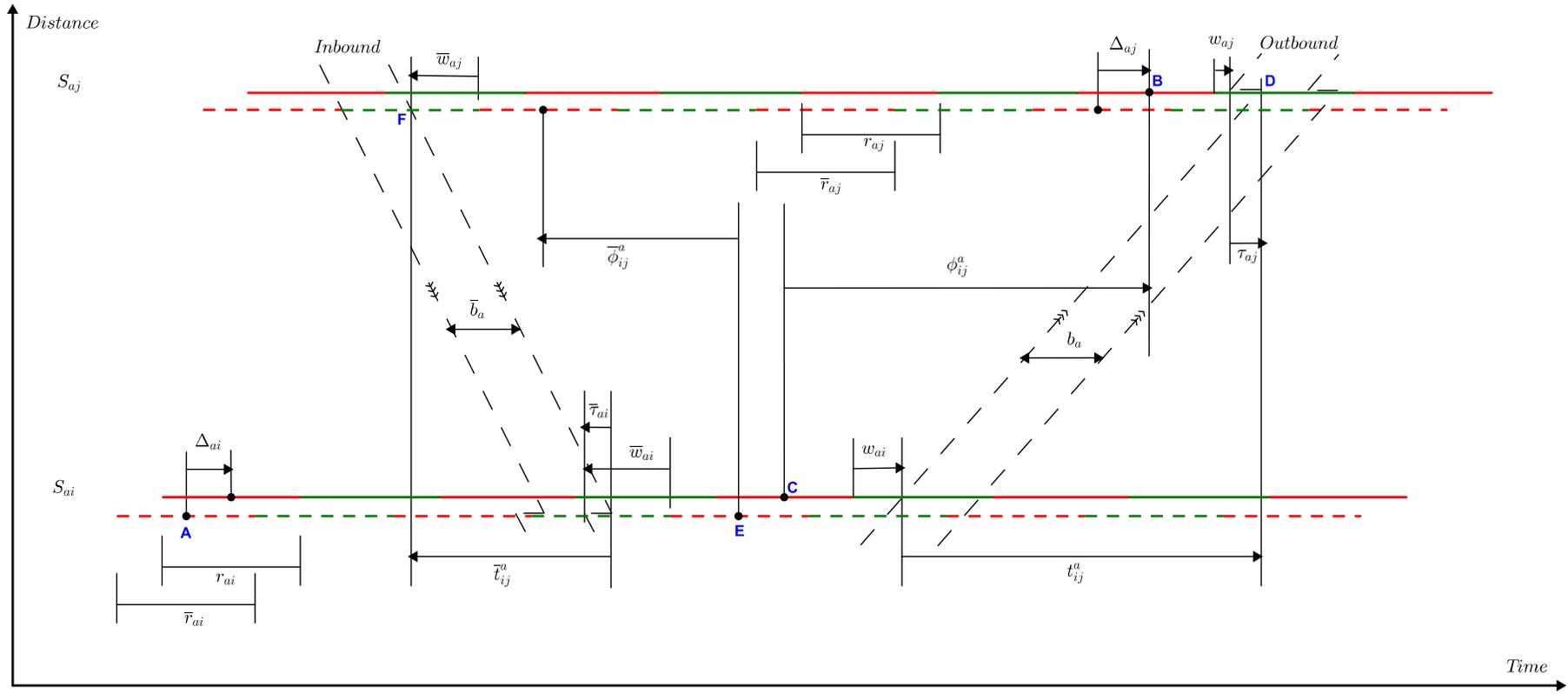}
  \caption{Geometry for MAXBAND model on artery $a$.}
  \label{ggm}
\end{figure}
\end{landscape}

\begin{linearmodel}[h]
\caption{MAXBAND, Maximal Bandwidth Formulation}
\label{maxbandn}

\footnotesize

\begin{align}
\text{Maximize}\hspace{4.1cm}  \sum_{\mathclap{a \in A}} \left( k_ab_{a}+\overline{k}_a\overline{b}_{a}\right) \hspace{4.4cm} \label{obj}
\end{align}

$\hspace{1.8cm}  \text{subject to}:$

\begin{align}
\frac{1}{T_2} \leq z \leq \frac{1}{T_1},& &  \label{lperiod} \\[0.05em]
w_{ai}+b_a \leq  1-r_{ai},& & \forall a \in A, \hspace{0.05cm} \forall i=1,\ldots,n_a, \label{greenc1}\\[0.1em]
\overline{w}_{ai}+\overline{b}_a \leq  1-\overline{r}_{ai},& & \forall a \in A, \hspace{0.05cm}  \forall i=1,\ldots,n_a, \label{greenc2} \\[0.1em]
(w_{ai}+\overline{w}_{ai})-(w_{a,i+1}+\overline{w}_{a,i+1})+(t_{i}^a+\overline{t}_{i}^a) \nonumber\\
+(\delta_{ai}\ell_{ai}-\overline{\delta}_{ai}\overline{\ell}_{ai})-(\delta_{a,i+1}\ell_{a,i+1}-\overline{\delta}_{a,i+1}\overline{\ell}_{a,i+1})+(r_{ai}-r_{a,i+1}) \nonumber\\
-(\tau_{a,i+1}+\overline{\tau}_{ai})=m_{i}^a,& & \forall a \in A, \hspace{0.05cm} \forall i=1,\ldots,n_a-1, \label{mbcompletec} \\[0.1em]
\left(\ \frac{d_{i}^a}{f_{i}^a}\right)z \leq t_{i}^a \leq  \left(\ \frac{d_{i}^a}{e_{i}^a}\right)z, & & \forall a \in A, \hspace{0.05cm} \forall i=1,\ldots,n_a-1, \label{vone} \\[0.1em]
\left(\ \frac{\overline{d}_{i}^a}{\overline{f}_{i}^a}\right)z \leq \overline{t}_{i}^a \leq  \left(\ \frac{\overline{d}_{i}^a}{\overline{e}_{i}^a}\right)z,& & \forall a \in A, \hspace{0.05cm} \forall i=1,\ldots,n_a-1, \label{vtwo} \\[0.1em]
\left(\ \frac{d_{i}^a}{h_{i}^a}\right)z \leq \left(\ \frac{d_{i}^a}{d_{i+1}^a}\right)t_{i+1}^a-t_{i}^a \leq  \left(\ \frac{d_{i}^a}{g_{i}^a}\right)z,& & \forall a \in A, \hspace{0.05cm} \forall i=1,\ldots,n_a-2, \label{rvone} \\[0.1em]
\left(\ \frac{\overline{d}_{i}^a}{\overline{h}_{i}^a}\right)z \leq \left(\ \frac{\overline{d}_{i}^a}{\overline{d}_{i+1}^a}\right)\overline{t}_{i+1}^a-\overline{t}_{i}^a \leq  \left(\ \frac{\overline{d}_{i}^a}{\overline{g}_{i}^a}\right)z,& & \forall a \in A, \hspace{0.05cm} \forall i=1,\ldots,n_a-2, \label{rvtwo} \\[0.1em]
\sum_{\mathclap{(i,j): a \in A_\zeta^F}} \phi_{ij}^a \hspace{0.2cm} - \hspace{0.2cm} \sum_{\mathclap{(i,j): a \in A_\zeta^B}} \phi_{ij}^a \hspace{0.2cm} + \hspace{0.2cm} \sum_{\mathclap{(b,j,i,c,k) \in J_\zeta}}\Psi_{S_{bj},S_{ck}}^i=C_{\zeta}, && \forall \zeta \in \mathcal{B}_\zeta, \label{cycleclm} \\[0.1em]
C_{\zeta} \in \mathbb{Z}, & & \forall \zeta \in \mathcal{B}_\zeta, \\[0.1em]
m_{i}^a \in \mathbb{Z}, & & \forall a \in A, \hspace{0.05cm} \forall i=1,\ldots,n_a-1, \\[0.1em]
\delta_{ai}, \overline{\delta}_{ai} \in \{0,1\}, & & \forall a \in A, \hspace{0.05cm} \forall i=1,\ldots,n_a-1, \\[0.1em]
b_a, \overline{b}_a, t_{i}^a,\overline{t}_{i}^a, w_{ai}, \overline{w}_{ai},z \geq 0, & & \forall a \in A, \hspace{0.05cm} \forall i=1,\ldots,n_a-1.
\end{align}

\end{linearmodel}

Furthermore, we consider only networks that can be represented by two dimensional grid graphs $G_{r \times c}(V,E)$, where $r$ and $c$ are the number of rows and columns, respectively. $V$ is the set of nodes, $E$ is the set of edges, $|V|=n=rc$ and $|E|=m=2rc-r-c$. Grid graphs are a good representation of many real-world networks. An example is shown in Figure \ref{ggexample} with $n=12$ and $m=17$.

\begin{figure}[h!]
  \centering
    \includegraphics[width=0.35\textwidth]{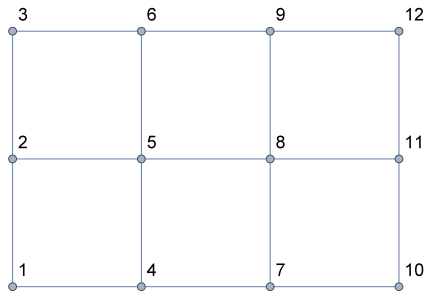}
  \caption{A grid graph $G_{3 \times 4}(V,E)$.}
  \label{ggexample}
\end{figure}

\subsection{Objective Function}

In terms of level of vehicular congestion, in real urban networks some streets or avenues can be more important than others. If this is the case, these important arteries must be explicitly prioritized in the model. This can be reached by giving different weights on every artery in the network. This discrimination of arteries will produce a decrease in the efficiency of the use of traffic lights in those less important in favour of those weighted more highly. Therefore, the calibration of these parameters must be careful and based on traffic volume statistics.

Let $A$ be the set of arteries on the network. We define $k_a$ and $\overline{k}_a$ as the weights for the outbound and inbound bandwidth at artery $a$, respectively. The objective function to be maximized is then:

\begin{equation*}
\sum_{\mathclap{a \in A}} \left( k_ab_{a}+\overline{k}_a\overline{b}_{a}\right).
\end{equation*}

\subsection{Arterial Constraints} \label{sarterialc}

As said before, we assume that all signals work into a common signal period whose length is introduced in the model as a decision variable. It must lie in an interval $[T_1,T_2]$, see $(\ref{lperiod})$. Decision variable $z$ is the reciprocal of the period length, that is, $z=1/T$. Inequalities $(\ref{greenc1})$-$(\ref{greenc2})$ ensure that the bandwidth remains within the green time. The velocities $v_i^a$ between each signal on each artery are decision variables and bounded, with $e_i^a$ and $f_i^a$ representing the lower and upper limits, respectively, and $d_i^a$ is the distance between two consecutive arteries, see $(\ref{vone})$-$(\ref{vtwo})$. In order to avoid sudden changes in the velocities between consecutive signals, they are limited by imposing lower and upper bounds $1/h_i^a$ and $1/g_i^a$ on changes in reciprocal velocities. The reason for using reciprocal bounds for velocities and changes in velocities is that linear constraints can be obtained in this way. It is not possible to consider directly the inequalities $e_i^a \leq v_i^a \leq f_i^a$ because the period length is also a variable. Thus, if we try to pass from $v_i^a$ in meters/second to $v_i^a$ in meters/period, we will have $e_i^aT \leq v_i^aT \leq f_i^aT$, which are nonlinear constraints. Therefore, we use the reciprocals of $e_i^a$ and $f_i^a$,

\begin{equation*}
e_i^a \leq v_i^a \leq f_i^a \quad \to \quad \frac{d_i^a}{f_i^a} \leq \frac{d_i^a}{v_i^a} \leq \frac{d_i^a}{e_i^a} \quad \to \quad \frac{d_i^a}{f_i^a}z \leq t_i^a \leq \frac{d_i^a}{e_i^a}z.
\end{equation*}

The same applies to the changes of velocities.\\

It can be seen in Figure \ref{ggm} that $Time_{\textsf{A-B}}=\Delta_{ai}+$integer number of periods$+\phi_{ij}^a$ and that $Time_{\textsf{A-B}}=$integer number of periods$-\overline{\phi}_{ij}^a+$integer number of periods$+\Delta_{aj}$. Therefore:

\begin{equation} \label{basic3}
\phi_{ij}^a+\overline{\phi}_{ij}^a+\Delta_{ai}-\Delta_{aj}=m_{ij}^a,
\end{equation}

\noindent where $m_{ij}^a$ is an integer decision variable (number of periods). Also, $Time_{\textsf{C-D}}=\phi_{ij}^a+\frac{1}{2}r_{aj}+w_{aj}+\tau_{aj}=\frac{1}{2}r_{ai}+w_{ai}+t_{ij}^a$ and $Time_{\textsf{E-F}}=\overline{\phi}_{ij}^a+\frac{1}{2}\overline{r}_{aj}+\overline{w}_{aj}=\frac{1}{2}\overline{r}_{ai}+\overline{w}_{ai}-\overline{\tau}_{ai}+\overline{t}_{ij}^a$. So, if we now substitute in (\ref{basic3}), we have that

\begin{equation} \label{basic4a}
\begin{multlined}
t_{ij}^a+\overline{t}_{ij}^a+\frac{1}{2}(r_{ai}+\overline{r}_{ai})+(w_{ai}+\overline{w}_{ai})-\frac{1}{2}(r_{aj}+\overline{r}_{aj})-(w_{aj}+\overline{w}_{aj})-(\tau_{aj}+\overline{\tau}_{ai})\\
+(\Delta_{ai}-\Delta_{aj})=m_{ij}^a.
\end{multlined}
\end{equation}

Equation (\ref{basic4a}) is named \emph{arterial loop constraint} for artery $a$ between signals $S_{ai}$ and $S_{aj}$.

\begin{figure}[h]
  \centering
    \includegraphics[width=0.8\textwidth]{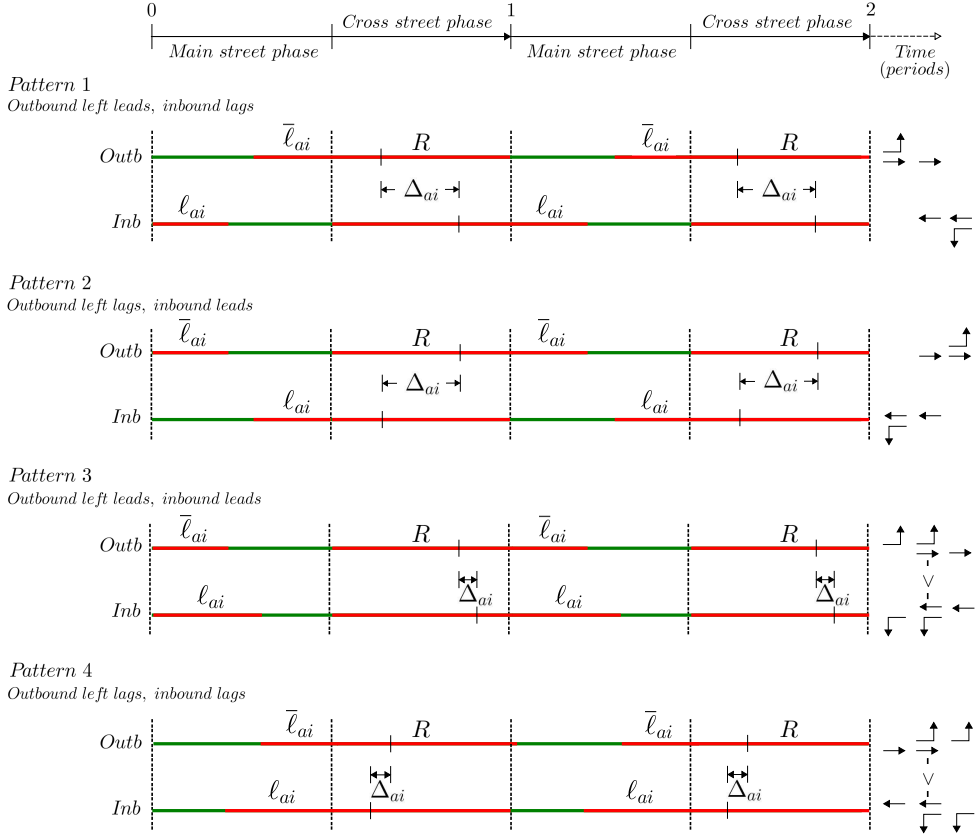}
  \caption{Patterns of left turn phases.}
  \label{ltpnew}
\end{figure}

In addition, there are constraints that model left turn decisions if it is allowed by green lights. The MAXBAND allows to decide among four possible patterns of left turns which can be seen in Figure \ref{ltpnew} (see \cite{Little81} for more details).

Parameters $\ell_{ai}$ and $\overline{\ell}_{ai}$ in Figure \ref{ltpnew} represent, for a signal $i$ on an artery $a$, the time (periods) of outbound and inbound left turn phases respectively. $R$ is the common red time. For instance, Figure \ref{ltpexample} shows the three possible movements for vehicles on a main street in three different moments. As can be seen at area 2, the traffic lights are green for outbound and inbound directions, so no car in the horizontal street can cross to the other side. On the common red time $R$ a possible different left turn pattern can be given for cross street.

\begin{figure}[h]
  \centering
    \includegraphics[width=0.5\textwidth]{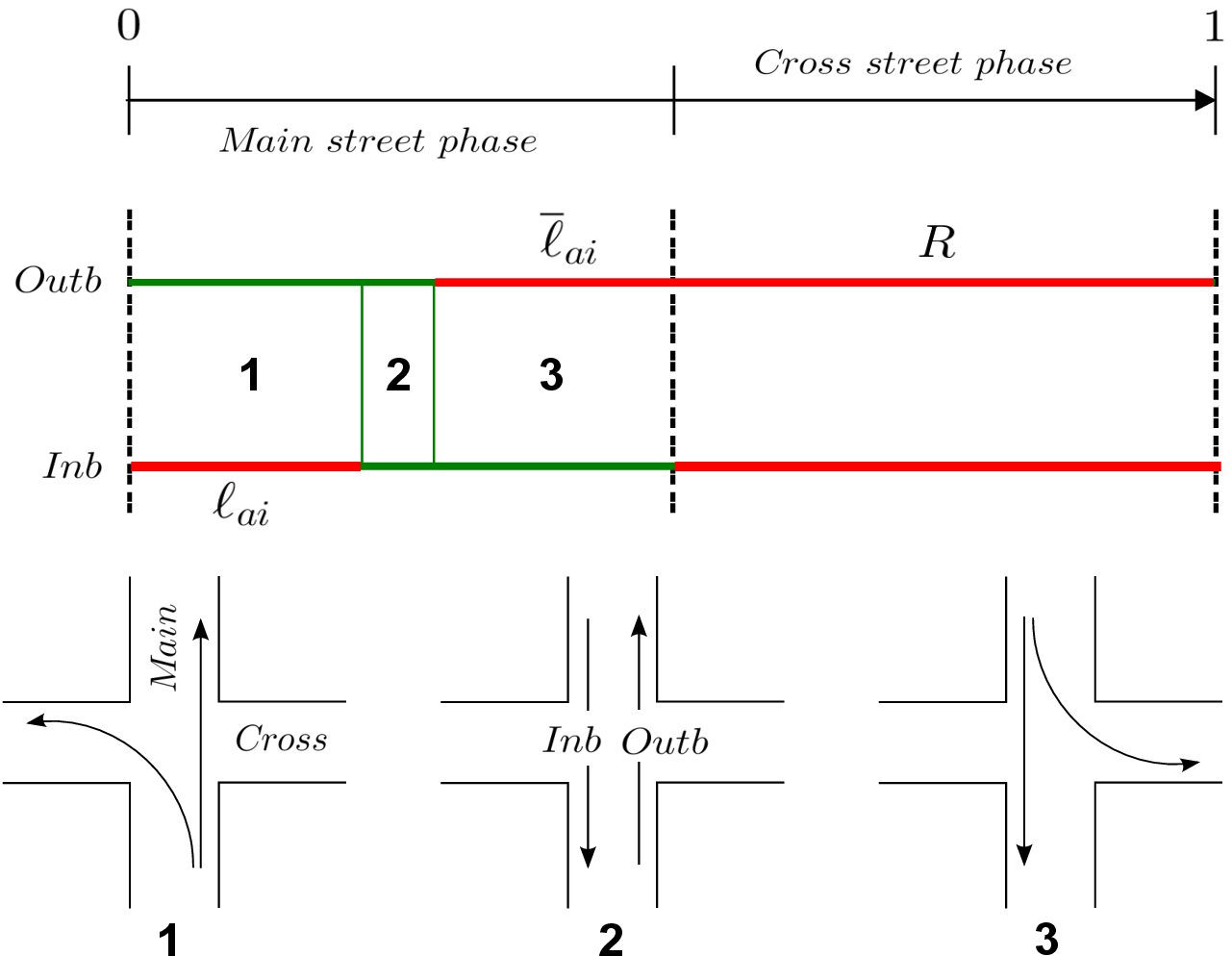}
  \caption{Left turn phase example with Pattern 1.}
  \label{ltpexample}
\end{figure}

Furthermore, $\Delta_{ai}$ can be expressed as a function of $\ell_{ai}$ and $\overline{\ell}_{ai}$. For example, if we consider Pattern 1 and we calculate the difference between the center of total red time of outbound and the total red of inbound (in that order), we have that

\begin{equation*}
\Delta_{ai}=\frac{\overline{\ell}_{ai}+R}{2}-\left( \frac{R+\ell_{ai}}{2}+\overline{\ell}_{ai} \right)=-\frac{\ell_{ai}+\overline{\ell}_{ai}}{2}.
\end{equation*}

The results for the other left turn phases are shown in Table \ref{deltas}. All these expressions can be obtained with the formula:

\begin{equation} \label{deltascomputation}
\Delta_{ai}=\frac{1}{2}[(2\delta_{ai}-1)\ell_{ai}-(2\overline{\delta}_{ai}-1)\overline{\ell}_{ai}],
\end{equation}

\noindent where $\delta_{ai}$, $\overline{\delta}_{ai}$ $\in$ $\set{0,1}$ are additional binary variables. The decisions on left turns are included in the model by substituting equation (\ref{deltascomputation}) in (\ref{basic4a}), as can be seen in constraint $(\ref{mbcompletec})$.

\begin{table}[h]
\caption{Expressions for $\Delta_{ai}$'s.}
\label{deltas}\centering

\renewcommand{\arraystretch}{2.4} 
\setlength{\tabcolsep}{1.5em} 
  \footnotesize
\begin{tabular}{cccc}
\toprule
$Pattern$ & $\Delta_{ai}$ & $\delta_{ai}$ & $\overline{\delta}_{ai}$ \\
\midrule
1 & $\displaystyle -\frac{\ell_{ai}+\overline{\ell}_{ai}}{2}$ & 0 & 1 \\
2 & $\displaystyle \frac{\ell_{ai}+\overline{\ell}_{ai}}{2}$ & 1 & 0 \\
3 & $\displaystyle -\frac{\ell_{ai}-\overline{\ell}_{ai}}{2}$ & 0 & 0 \\
4 & $\displaystyle \frac{\ell_{ai}-\overline{\ell}_{ai}}{2}$ & 1 & 1 \\
\bottomrule
\end{tabular}
\end{table}

\subsection{Loop Constraints}

The network case is a natural generalization of the arterial case and the corresponding model includes all the constraints for each artery as shown before. The arterial loop constraints (\ref{basic4a}) can be seen as a cycle for two nodes because it represents the movement of going to and returning from a signal. If we extend this idea to larger cycles, it is clear that the sum of all the offsets in the cycle must be an integer number as well.

\begin{figure}[h]
  \centering
    \includegraphics[width=0.25\textwidth]{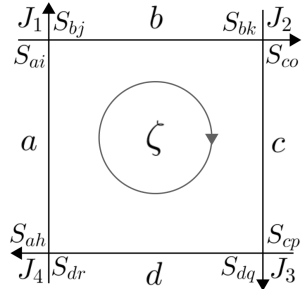}
  \caption{Clockwise loop with 4 junctions.}
  \label{c4j}
\end{figure}

In order to see how to write the equation of the cycle constraints, let us start with an example. A cycle with $4$ arteries $A=\{a,b,c,d\}$ and $4$ junctions $J=\{J_1,J_2,J_3,J_4\}$ is shown in Figure \ref{c4j}. Each artery $a \in A$ has signals $S_{aj}$, where $j$ is the index for signals on $a$ increasing in the outbound direction given by the arrows. Heading in the clockwise direction and starting from junction $J_1$, the cycle constraint for this example is: 

\begin{equation*}
\phi_{jk}^b+\Psi_{S_{bk},S_{co}}^{J_2}+\phi_{op}^c+\Psi_{S_{cp},S_{dq}}^{J_3}+\phi_{qr}^d+\Psi_{S_{dr},S_{ah}}^{J_4}+\phi_{hi}^a+\Psi_{S_{ai},S_{bj}}^{J_1}=C_{\zeta,}
\end{equation*}

\noindent where $C_{\zeta}$ is an integer decision variable and $\Psi_{S_{aj},S_{bk}}^i$ is a decision variable named \emph{intranode offset} which represents the time between consecutive centers of reds for signals $S_{aj}$ and $S_{bk}$ that meet at junction $i$, i.e, it is a link time between arteries $a$ and $b$.

In order to generalize the previous expression to any cycle, we define the following sets:

\begin{itemize}
\item $A_\zeta^F$ ($A_\zeta^B$): Set of all segments of forward (backward) arteries with edges $(i,j)$ in the clockwise direction of cycle $\zeta$,
\item $J_\zeta$: All sets of the form $(b,j,i,c,k)$ in $\zeta$, where $i$ is the junction between arteries~$b$ and $c$ in the signals $S_{bj}$ and $S_{ck}$,
\end{itemize}

Then, the \emph{network loop constraint} (cycle constraint) is: 

\begin{equation*}
\sum_{\mathclap{(i,j): a \in A_\zeta^F}} \phi_{ij}^a \hspace{0.2cm} - \hspace{0.2cm} \sum_{\mathclap{(i,j): a \in A_\zeta^B}} \phi_{ij}^a \hspace{0.2cm} + \hspace{0.2cm} \sum_{\mathclap{(b,j,i,c,k) \in J_\zeta}}\Psi_{S_{bj},S_{ck}}^i=C_{\zeta}.
\end{equation*}

The number of cycle constraints in the model will depend on how many edges and nodes the network has. Unfortunately, this number can be very high, which makes the problem very difficult to solve. The following result helps us to alleviate this problem: It is well known that the set of all cycles $\zeta$ on any single graph can be spanned by a basis $\mathcal{B}_\zeta$ with cardinality $m-n+1$, where $m$ represents the number of edges and $n$ the number of nodes on the underlying undirected graph related to the directed graph that represents the original network. See \cite{LandR} for full details. So, a cycle basis must be found before writing down the model.

\subsection{Computing Intranode Offset} \label{intranode}

In general, a grid graph $G_{k \times k}$ needs $(k-1)^2$ network loop constraints and therefore several intranodes must be computed for each of these equations. The values of intranode offsets $\Psi_{S_{aj},S_{bk}}^i$'s depend on red times positions of signals $S_{aj}$ and $S_{bk}$ on a main and cross street, respectively. For instance, if left turns are not permitted, then the red and green times for main and cross street will have the same length and in this case the intranode offset is clearly 0.5 periods. Since the MAXBAND model must decide among four different left turn patterns on each junction $i$, the computation of each intranode should take into account all the possible values of binary variables involved in that choice. The next result provides a simple expression to compute intranode offsets which does not require of extra variables beyond those used in model LM \ref{maxbandn}.

\begin{theorem}

Consider the patterns of left turn phases shown in Figure \ref{ltpnew}. Let $\psi_{mc}^{p_mp_c}$ be the value of $\Psi_{S_{mj},S_{ck}}^i$ when arteries $m$ and $c$  meet at junction $i$ for signals $S_{mj}$ and $S_{ck}$ with left turn phases patterns $p_m$ and $p_c$ respectively. Then, for all possible values of $p_m$ and $p_c$, we have that:

\begin{equation*}
\Psi_{S_{mj},S_{ck}}^i=\frac{1}{2}-\frac{1}{2}\left[(2\overline{\delta}_{ck}-1)\overline{\ell}_{ck}-(2\overline{\delta}_{mj}-1)\overline{\ell}_{mj}\right].
\end{equation*}

\end{theorem}

\begin{proof}

Let us consider Figure \ref{maincross}.

\begin{figure}[H]
  \centering
    \includegraphics[width=0.25\textwidth]{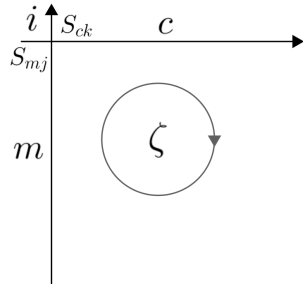}
  \caption{A junction $i$ of arteries main $(m)$ and cross $(c)$.}
  \label{maincross}
\end{figure}

Figure \ref{psi} shows the different forms that $\psi_{mc}^{p_mp_c}$ may have for all possible permutations of left turn phases in Figure \ref{ltpnew}. The cross street phase takes place during the red time $R_m$ in the main street and vice versa.

\begin{figure}[h]
  \centering
    \includegraphics[width=0.90\textwidth]{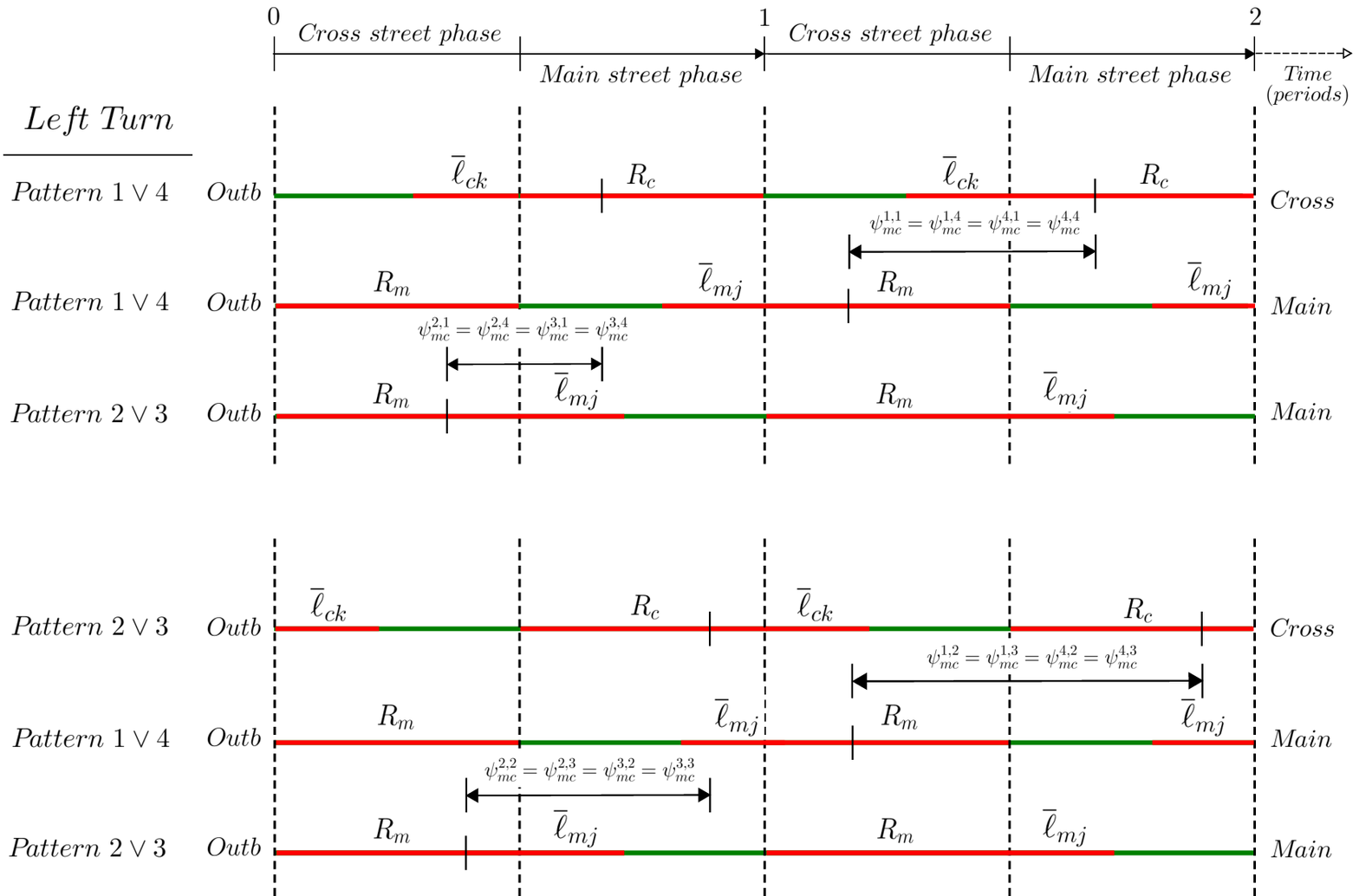}
  \caption{Geometry for $\displaystyle \psi_{mc}^{p_mp_c}$.}
  \label{psi}
\end{figure}

\begin{table}[h]
\caption{Expressions for $\psi_{mc}^{p_mp_c}$'s.}
\label{fpsi}\centering

\setlength{\tabcolsep}{8.5pt} 
\renewcommand{\arraystretch}{2.0} 
  \footnotesize
\begin{tabular}{cccccc}
\toprule
& & \multicolumn{4}{c}{Cross Street} \\
\cmidrule(l){3-6}
& Patterns & $1$ & $2$ & $3$ & $4$ \\
\midrule

\parbox[t]{1mm}{\multirow{4}{*}{\rotatebox[origin=c]{90}{Main Street}}}& $1$ & $\displaystyle \frac{1-\overline{\ell}_{ck}+\overline{\ell}_{mj}}{2}$ & $\displaystyle \frac{1+\overline{\ell}_{ck}+\overline{\ell}_{mj}}{2}$ & $\displaystyle \frac{1+\overline{\ell}_{ck}+\overline{\ell}_{mj}}{2}$ & $\displaystyle \frac{1-\overline{\ell}_{ck}+\overline{\ell}_{mj}}{2}$\\
& $2$ & $\displaystyle \frac{1-\overline{\ell}_{ck}-\overline{\ell}_{mj}}{2}$ & $\displaystyle \frac{1+\overline{\ell}_{ck}-\overline{\ell}_{mj}}{2}$ & $\displaystyle \frac{1+\overline{\ell}_{ck}-\overline{\ell}_{mj}}{2}$ & $\displaystyle \frac{1-\overline{\ell}_{ck}-\overline{\ell}_{mj}}{2}$ \\
& $3$ & $\displaystyle \frac{1-\overline{\ell}_{ck}-\overline{\ell}_{mj}}{2}$ & $\displaystyle \frac{1+\overline{\ell}_{ck}-\overline{\ell}_{mj}}{2}$ & $\displaystyle \frac{1+\overline{\ell}_{ck}-\overline{\ell}_{mj}}{2}$ & $\displaystyle \frac{1-\overline{\ell}_{ck}-\overline{\ell}_{mj}}{2}$\\
& $4$ & $\displaystyle \frac{1-\overline{\ell}_{ck}+\overline{\ell}_{mj}}{2}$ & $\displaystyle \frac{1+\overline{\ell}_{ck}+\overline{\ell}_{mj}}{2}$ & $\displaystyle \frac{1+\overline{\ell}_{ck}+\overline{\ell}_{mj}}{2}$ & $\displaystyle \frac{1-\overline{\ell}_{ck}+\overline{\ell}_{mj}}{2}$\\
\bottomrule
\end{tabular}
\end{table}

All values of $\psi_{mc}^{p_mp_c}$ are summarized in Table \ref{fpsi}. We take into account only the outbound directions phases because we are using only the $\phi$'s in equation $(\ref{cycleclm})$ of the MAXBAND model.

Also, in Table \ref{fpsi} it can be seen that:

\noindent $\displaystyle \psi_{mc}^{1,1}=\psi_{mc}^{1,4}=\psi_{mc}^{4,1}=\psi_{mc}^{4,4}=\frac{1-\overline{\ell}_{ck}+\overline{\ell}_{mj}}{2}$,
$\displaystyle \psi_{mc}^{2,1}=\psi_{mc}^{2,4}=\psi_{mc}^{3,1}=\psi_{mc}^{3,4}=\frac{1-\overline{\ell}_{ck}-\overline{\ell}_{mj}}{2}$,
$\displaystyle \psi_{mc}^{1,2}=\psi_{mc}^{1,3}=\psi_{mc}^{4,2}=\psi_{mc}^{4,3}=\frac{1+\overline{\ell}_{ck}+\overline{\ell}_{mj}}{2}$, and
$\displaystyle \psi_{mc}^{2,2}=\psi_{mc}^{2,3}=\psi_{mc}^{3,2}=\psi_{mc}^{3,3}=\frac{1+\overline{\ell}_{ck}-\overline{\ell}_{mj}}{2}$.
 
\begin{table}[h!]
\caption{Four different groups of $\displaystyle \psi_{mc}^{p_mp_c}$'s on junction $i$.}
\label{ppsi}\centering
\renewcommand{\arraystretch}{1.2}

\setlength{\tabcolsep}{0.4cm} 
  \footnotesize
\begin{tabular}{ccc>{\columncolor[gray]{0.9}}cc>{\columncolor[gray]{0.9}}cc}
\toprule
\multicolumn{2}{c}{Patterns}  &  &  &  &  &  \\
$p_m$ & $p_c$ & $\delta_{mj}$ & $\overline{\delta}_{mj}$ & $\delta_{ck}$ & $\overline{\delta}_{ck}$ & $\displaystyle \psi_{mc}^{p_mp_c}$ \\
\midrule
$4$&$4$&1&1&1&1&$\displaystyle \psi_{mc}^{4,4}$\\
$4$&$1$&1&1&0&1&$\displaystyle \psi_{mc}^{4,1}$\\
$1$&$4$&0&1&1&1&$\displaystyle \psi_{mc}^{1,4}$\\
$1$&$1$&0&1&0&1&$\displaystyle \psi_{mc}^{1,1}$\\
\midrule
$3$&$4$&0&0&1&1&$\displaystyle \psi_{mc}^{3,4}$\\
$2$&$1$&1&0&0&1&$\displaystyle \psi_{mc}^{2,1}$\\
$3$&$1$&0&0&0&1&$\displaystyle \psi_{mc}^{3,1}$\\
$2$&$4$&1&0&1&1&$\displaystyle \psi_{mc}^{2,4}$\\
\midrule
$1$&$2$&0&1&1&0&$\displaystyle \psi_{mc}^{1,2}$\\
$4$&$2$&1&1&1&0&$\displaystyle \psi_{mc}^{4,2}$\\
$4$&$3$&1&1&0&0&$\displaystyle \psi_{mc}^{4,3}$\\
$1$&$3$&0&1&0&0&$\displaystyle \psi_{mc}^{1,3}$\\
\midrule
$2$&$2$&1&0&1&0&$\displaystyle \psi_{mc}^{2,2}$\\
$2$&$3$&1&0&0&0&$\displaystyle \psi_{mc}^{2,3}$\\
$3$&$2$&0&0&1&0&$\displaystyle \psi_{mc}^{3,2}$\\
$3$&$3$&0&0&0&0&$\displaystyle \psi_{mc}^{3,3}$\\
\bottomrule
\end{tabular}
\end{table}
 
In Table \ref{ppsi} all different values of the binary variables $\delta$'s and $\overline{\delta}$'s are shown for each left turn phase that the model uses to compute the $\Delta$'s with equation (\ref{deltascomputation}). The $\psi_{mc}^{p_mp_c}$'s are arranged in four groups determined just for the values of $\overline{\delta}$'s on the signals  $S_{mj}$ and $S_{ck}$.

It is now easy to verify that a single expression to compute any $\Psi_{S_{mj},S_{ck}}^i$ is given by:

\begin{equation} \label{psiscomputation}
\Psi_{S_{mj},S_{ck}}^i=\frac{1}{2}-\frac{1}{2}\left[(2\overline{\delta}_{ck}-1)\overline{\ell}_{ck}-(2\overline{\delta}_{mj}-1)\overline{\ell}_{mj}\right].
\end{equation} 

\end{proof}

A similar result that considers pedestrian crossing times can be found in \cite{Chaudhary87}. However it must be noted the result that is shown here was developed independently.

\section{Bounds for Integer Variables}\label{s4}

\cite{Little66} provides bounds for the integer variables in the arterial loop constraints when left turns phases are not included. In this section we extend those bounds for the integer variables in the network loop constraints.

It is clear from Figure \ref{ggm} that $0 \leq w_{ai} \leq 1-r_{ai}$ and that $0 \leq \overline{w}_{ai} \leq 1-\overline{r}_{ai}$ for any signal $i$ on artery $a$. By using constraints (\ref{lperiod}), (\ref{vone}) and (\ref{vtwo}) we have that $\frac{d_{ij}^a}{f_{ij}^aT_2} \leq t_{ij}^a \leq  \frac{d_{ij}^a}{e_{ij}^aT_1}$ and that $\frac{d_{ij}^a}{\overline{f}_{ij}^aT_2} \leq \overline{t}_{ij}^a \leq   \frac{d_{ij}^a}{\overline{e}_{ij}^aT_1}$. Furthermore, equation (\ref{deltascomputation}) gives $ \Delta_{ai}=\delta_{ai}\ell_{ai}-\frac{\ell_{ai}}{2}-\overline{\delta}_{ai}\overline{\ell}_{ai}+\frac{\overline{\ell}_{ai}}{2}$. Therefore, we have that $- (\frac{\ell_{ai}}{2}+\frac{\overline{\ell}_{ai}}{2}) \leq \Delta_{ai} \leq (\frac{\ell_{ai}}{2}+\frac{\overline{\ell}_{ai}}{2})$, as $\delta$'s take values in $\{0,1\}$.

If we now use these bounds in constraints $(\ref{mbcompletec})$, then we have the following limits for $m$'s in terms of the MAXBAND model's parameters, $\underline{m_{ij}^a} \leq m_{ij}^a \leq \overline{m_{ij}^a}$, where:

\begin{equation} \label{Um}
\begin{multlined}
\overline{m_{ij}^a}  = \left \lfloor 2-\frac{1}{2}(r_{ai}+\overline{r}_{ai})-\frac{1}{2}(r_{aj}+\overline{r}_{aj})+\frac{1}{2}(\ell_{ai}+\overline{\ell}_{ai})+\frac{1}{2}(\ell_{aj}+\overline{\ell}_{aj}) \right.\\ \left. -(\tau_{aj}+\overline{\tau}_{ai})+\frac{d_{ij}^a}{e_{ij}^aT_1} +\frac{d_{ij}^a}{\overline{e}_{ij}^aT_1} \right \rfloor
\end{multlined},
\end{equation}

\begin{equation} \label{Lm}
\begin{multlined}
\underline{m_{ij}^a}  =  \left \lceil -2+\frac{1}{2}(r_{ai}+\overline{r}_{ai})+\frac{1}{2}(r_{aj}+\overline{r}_{aj})-\frac{1}{2}(\ell_{ai}+\overline{\ell}_{ai})-\frac{1}{2}(\ell_{aj}+\overline{\ell}_{aj}) \right.\\ \left. -(\tau_{aj}+\overline{\tau}_{ai})+\frac{d_{ij}^a}{f_{ij}^aT_2} +\frac{d_{ij}^a}{\overline{f}_{ij}^aT_2}  \right \rceil.
\end{multlined}
\end{equation}

Therefore, following the notation in LM \ref{maxbandn}, i.e., $m_{i,i+1}^a=m_{i}^a$, we have the following bounds:

\begin{equation} \label{limitsm}
\underline{m_{i}^a} \leq m_{i}^a \leq \overline{m_{i}^a},\quad \forall a \in A, \hspace{0.05cm} \forall i=1,\ldots,n_a-1.
\end{equation}

Furthermore, let $a \in A_\zeta^F$ be, as in previous sections, an artery with signals $S_{ai}$, where $i \in I_\zeta^a=\set{1_\zeta^a, \ldots , n_\zeta^a}$ and increasing in the outbound direction. The time between the first signal and the last one in the artery segment is:

\begin{equation} \label{tincycle}
t_{1_\zeta^a,n_\zeta^a}^a=\sum_{\mathclap{i \in I_\zeta^a \setminus \set{n_\zeta^a}}} t_i^a \hspace{0.2cm} - \hspace{0.2cm} \sum_{\mathclap{i \in I_\zeta^a \setminus \set{1_\zeta^a}}} \tau_{ai}.
\end{equation}

Therefore, $\underline{t_{1_\zeta^a,n_\zeta^a}^a} \leq t_{1_\zeta^a,n_\zeta^a}^a \leq \overline{t_{1_\zeta^a,n_\zeta^a}^a}$, where:

\begin{equation} \label{ulimitt}
\overline{t_{1_\zeta^a,n_\zeta^a}^a}=\sum_{\mathclap{i \in I_\zeta^a \setminus \set{n_\zeta^a}}} \hspace{0.3cm} \frac{d_{i}^a}{e_{i}^aT_1} \hspace{0.2cm} - \hspace{0.2cm} \sum_{\mathclap{i \in I_\zeta^a \setminus \set{1_\zeta^a}}} \tau_{ai},
\end{equation}

\begin{equation} \label{llimitt}
\underline{t_{1_\zeta^a,n_\zeta^a}^a}=\sum_{\mathclap{i \in I_\zeta^a \setminus \set{n_\zeta^a}}} \hspace{0.3cm} \frac{d_{i}^a}{f_{i}^aT_2} \hspace{0.2cm} - \hspace{0.2cm} \sum_{\mathclap{i \in I_\zeta^a \setminus \set{1_\zeta^a}}} \tau_{ai}.
\end{equation}

By using the same facts mentioned above and since $\phi_{ij}^a+\frac{1}{2}r_{aj}+w_{aj}+\tau_{aj}=\frac{1}{2}r_{ai}+w_{ai}+t_{ij}^a$, we obtain that $\underline{\phi_{i_\zeta^a,n_\zeta^a}^a} \leq \phi_{i_\zeta^a,n_\zeta^a}^a \leq \overline{\phi_{i_\zeta^a,n_\zeta^a}^a}$, with:

\begin{equation} \label{Uphi}
\overline{\phi_{i_\zeta^a,n_\zeta^a}^a} = -\frac{1}{2}(r_{a,1_\zeta^a}+r_{a,n_\zeta^a})+\overline{t_{1_\zeta^a,n_\zeta^a}^a}+1,
\end{equation}

\begin{equation} \label{Lphi}
\underline{\phi_{i_\zeta^a,n_\zeta^a}^a} = \frac{1}{2}(r_{a,1_\zeta^a}+r_{a,n_\zeta^a})+\underline{t_{1_\zeta^a,n_\zeta^a}^a}-1.
\end{equation}

According to equation (\ref{psiscomputation}), we also have that $\underline{\Psi_{S_{bj},S_{ck}}^i} \leq \Psi_{S_{bj},S_{ck}}^i \leq\overline{\Psi_{S_{bj},S_{ck}}^i}$:

\begin{equation} \label{UPsi}
\overline{\Psi_{S_{bj},S_{ck}}^i} = \frac{1}{2}(1+\overline{\ell}_{ck}+\overline{\ell}_{bj}),
\end{equation}

\begin{equation} \label{LPsi}
\underline{\Psi_{S_{bj},S_{ck}}^i} = \frac{1}{2}(1-\overline{\ell}_{ck}-\overline{\ell}_{bj}).
\end{equation}

Finally, bounds for $C_\zeta$ in the set of constrains (\ref{cycleclm}) are:

\begin{equation} \label{UC}
\overline{C_\zeta}=\left \lfloor  \hspace{0.4cm} \sum_{\mathclap{a \in A_\zeta^F}} \overline{\phi_{i_\zeta^a,n_\zeta^a}^a} \hspace{0.2cm} - \hspace{0.2cm} \sum_{\mathclap{a \in A_\zeta^B}} \underline{\phi_{i_\zeta^a,n_\zeta^a}^a} \hspace{0.2cm} + \hspace{0.2cm} \sum_{\mathclap{(b,j,i,c,j) \in J_\zeta}}\overline{\Psi_{S_{bj},S_{ck}}^i} \right \rfloor,
\end{equation}

\begin{equation} \label{LC}
\underline{C_\zeta}=\left \lceil \hspace{0.4cm} \sum_{\mathclap{a \in A_\zeta^F}} \underline{\phi_{i_\zeta^a,n_\zeta^a}^a} \hspace{0.2cm} - \hspace{0.2cm} \sum_{\mathclap{a \in A_\zeta^B}} \overline{\phi_{i_\zeta^a,n_\zeta^a}^a} \hspace{0.2cm} + \hspace{0.2cm} \sum_{\mathclap{(b,j,i,c,j) \in J_\zeta}}\underline{\Psi_{S_{bj},S_{ck}}^i} \right \rceil.
\end{equation}

The next set of constraints can be added to the linear model:

\begin{equation} \label{limitsC}
\underline{C_\zeta} \leq C_\zeta \leq \overline{C_\zeta},\quad \forall \zeta \in \mathcal{B}_\zeta.
\end{equation}

All these limits will be used for our computational experiments in the next section, in order to reduce the space of search of values of the integer variables.

\section{An MILP-Based Heuristic with Tabu Search for MAXBAND} \label{TSmethod}

In this section we propose a metaheuristic algorithm to solve the MAXBAND model with three variants. Then, we perform computational experiments and show the results.

On a network, the MAXBAND model requires a large number of initial data as well as several variables defined on each segment of an artery, signals and cycles in the cycle basis. Even though the instances presented in Table \ref{sizesP} may seem not very large, the formulation has many equalities which contain integer and binary variables. This makes the problem difficult to solve. 

\begin{table}[h]
  \centering
  \caption{Sizes of some MAXBAND problems.}
    \label{sizesP} \centering
    \renewcommand{\arraystretch}{1.2}
  \footnotesize    \begin{tabular}{rccccccccc}
    \toprule
    \multicolumn{1}{c}{} & \multicolumn{8}{c}{Grid Graph $G_{r \times c}$} \\
    \cmidrule(l){2-10}
    \multicolumn{1}{c}{} & 3x3   & 5x5   & 6x6   & 7x7   & 8x8   & 9x9   & 10x10 & 15x15 & 20x20 \\
    \midrule
    \multicolumn{1}{l}{Equalities} & 16    & 56    & 85    & 120   & 161   & 208   & 261   & 616   & 1121 \\
    \multicolumn{1}{l}{$m$'s} & 12    & 40    & 60    & 84    & 112   & 144   & 180   & 420   & 760 \\
    \multicolumn{1}{l}{$C$'s} & 4     & 16    & 25    & 36    & 49    & 64    & 81    & 196   & 361 \\
    \multicolumn{1}{l}{$\delta$'s} & 36    & 100   & 144   & 196   & 256   & 324   & 400   & 900   & 1600 \\
    \multicolumn{1}{l}{Integer variables} & 52    & 156   & 229   & 316   & 417   & 532   & 661   & 1516  & 2721 \\
    \bottomrule
    \end{tabular}%
\end{table}%

As will be seen later, a commercial MILP solver is not able to find the optimum for instances greater than $6\times6$ grid graph within a reasonable time (less than 3 hours). Therefore the use of heuristic methods is a logical alternative to the problem we are studying.	

One of the methods of solution for timing traffic lights is that used by TRANSYT, a commercial software that bases its solution on hill-climbing, evolutionary algorithm methods and simulation for modelling the behaviour and interactions of traffic flow. In \cite{Ratrout2014} and \cite{Lu2014} TRANSYT was compared with other similar software and showed to be more efficient than its competitors in terms of performance index (for example, number of vehicle stops) and determination of the common period length (red plus green light) for each signal on the network. In fact, some authors compare the results of their proposed approaches with those obtained using TRANSYT, see \cite{ThesisW} as an example. Regrettably, TRANSYT methodology has not been developed to be adapted to a problem of bandwidth maximization and therefore it is not an option to solve the MAXBAND model. This approach optimizes, instead of the bandwidth, a very complete objective function that involves delay, stops, fuel consumption, etc. thus most of the MAXBAND constraints are explicitly defined in it.

A heuristic alternative for MAXBAND was proposed by \cite{GartnerNew}. Their approach exploits the intrinsic geometry on the network. One iteration of this method consists of two general stages. The first one solves a subproblem from the original one with standard MILP techniques. The subproblem is a tree which is chosen because it contains no cycles. Moreover, this tree is not selected randomly as it considers either the streets with higher volume of traffic flow or some other measure of interest. This is why this tree is called \emph{priority arterial subnetwork}. After solving this reduced problem, the integer variables are fixed to the values obtained and used in a second stage to solve the whole problem. The procedure can be repeated if further improvements are required. If the original problem has a large size (dense when using all MAXBAND constraints), then the priority tree is difficult to solve, even though no integer cycle variables are used, the remaining variables could still be many. In fact, in \cite{GartnerNew} only two real but small cases were solved, the largest one is an incomplete $4\times4$ grid graph.

The method we propose is an adaptation of heuristic methods applicable for MAXBAND to try to solve larger instances.

The use of simulation helps the construction of multiple independent instances of a problem to obtain multiple initial feasible solutions, as required in evolutionary algorithms and how it is used in TRANSYT. Instead, the procedure we propose is based on a single feasible solution, as do heuristics such as Simulated Annealing \citep{Kirkpatrick1983} and Tabu Search \citep{GloverTS1}. We have decided to use a Tabu Search approach because its structure exploits intelligent strategies based on learning procedures that allows a guided search towards the optimal solution. Tabu Search uses a memory structure to generate new solutions from an initial one. These will compete among each other in each iteration of the procedure. Due to memory structure, changes applied in a solution to generate another one are considered in future iterations which, in favour of the diversity of the solutions, we would not like to repeat often. This avoids frequently repeating very complicated problems to be solved within the procedure. We use the variable fixing ideas to obtain new solutions in a short time. The method also involves a local search procedure to intensify the search of good solutions. It is done in three different ways that will be mentioned later.

\subsubsection*{The MILP-Based Heuristic}

As our algorithm is based on tabu search therefore it requires an initial solution whose neighbourhood needs to be explored. Despite many attempts to find a systematic procedure to generate the initial solution, this was not possible (due to the high number of equalities involved). Instead, we take as starting solution the first feasible solution found by the optimization solver (Xpress).

Next, we consider the set of the variables $m$'s ($2rc-r-c$), $\delta$'s ($4rc$) and $C$'s ($r(c-1)-c+1$) that must take integer values in the optimal solution. A number of $rm$, $rd$ and $rC$ variables are chosen randomly from them respectively to be modified later with one of the following procedures. The rest of these integer variables are fixed to the values that they have in the initial solution.

\begin{itemize}

\item TSILP-LSF: The values of $rm$, $rd$ and $rC$ variables are fixed to values within the bounds (\ref{limitsm}), $\set{0,1}$ and (\ref{limitsC}) respectively.

\item TSILP-LSU: The $rm$, $rd$ and $rC$ variables with values given by a solution become variables again (i.e., unfixed).

\item TSILP-LSVNS: TSILP-LSF and TSILP-LSU are applied one after the other.

\end{itemize}

The current problem is then solved with Xpress to obtain a new feasible neighbour solution. This is repeated in order to generate a set of solutions (\emph{candidate list}) of size \emph{SizeList}. The list of candidates may contain many infeasible solutions when using TSILP-LSF. If we unfix some variables, the number of infeasible solutions is reduced considerably. Indeed, we have used TSILP-LSU to generate the candidate list. Additionally, in our experience, if a memory structure is applied, procedures such as release (unfix) variables and solve the problem that remains, give a diverse enough set of solutions. Thus, TSILP-LSU is applied on the selected variables $rm$, $rd$ and $rC$ only if it is not forbidden by a \emph{tabu list}. This list is an array that contains $tt$ values for each variable which represent the number of iterations that they can not be modified. Then, this is sorted in decreasing order by the objective value (i.e., best first) and the the first solution becomes \emph{current solution}. It is saved and the \emph{tabu list} is updated for each integer variable. The last is done by decreasing $tt$ values if they are different from zero and by fixing them to a value $maxtt$ otherwise. Subsequently a greedy local search is applied to the \emph{current solution}. This is simply an iterative application of TSILP-LSVNS (Algorithm \ref{LSVNS}), TSILP-LSF (Algorithm \ref{LSF}) or TSILP-LSU (Algorithm \ref{LSU}). Finally this can be repeated until a maximum number of iterations is met. See Algorithm \ref{TSILP}.

\vspace{0.2cm}

VNS in TSILP-LSVNS stands for variable neighbourhood search, see \cite{VNS}. As can be seen later, the combination of fixing and unfixing random integer variables performs better.

\vspace{0.4cm}

\begin{algorithm}[h]
\caption{VNS Local Search Procedure (LSVNS).}
\label{LSVNS}
\vspace{0.50 cm}
\footnotesize

Let:\\
$S:$ A problem with best objective function value on a \emph{candidate list}.\\

Step 1. Choose $rm$, $rd$ and $rC$ from $m$'s, $\delta$'s and $C$'s on $S$.\\
Step 2. If $tt=0$, fix their values with random numbers within their bounds.\\
Step 3. Solve the instance using an LP solver (Xpress) and set this solution as a \emph{current solution}.\\
Step 4. Choose other integer variables $rm$, $rd$ and $rC$ from $m$'s, $\delta$'s and $C$'s on the \emph{current solution}.\\
Step 5. If $tt=0$, these variables are unfixed.\\
Step 6. Solve the instance using branch and bound (Xpress).\\
Step 7. Update the \emph{current solution} only if it is better than the previous one.\\
Step 8. Repeat the process until the maximum number of iterations is reached.\\
\end{algorithm}

\begin{algorithm}[h]
\caption{Fix local search procedure (LSF).}
\label{LSF}
\vspace{0.3 cm}
\footnotesize
Let:\\
$S:$ A problem with best objective function value on a \emph{candidate list}.\\

Steps 1 \ldots 3 and 7 \dots 8 are as in LSVNS.\\

\end{algorithm}

\begin{algorithm}[h]
\caption{Unfix local search procedure (LSU).}
\label{LSU}
\vspace{0.3 cm}
\footnotesize
Let:\\
$S:$ A problem with best objective function value on a \emph{candidate list}.\\

Steps 4 \ldots 6 and 7 \dots 8 are as in LSVNS.\\

\end{algorithm}

\begin{algorithm}[h]
\caption{Tabu Search for MAXBAND (TSILP-LS/F/U/VNS).}
\label{TSILP}
\vspace{0.50 cm}
\footnotesize

Let:\\
$P$: A complete MAXBAND problem on a grid graph.\\  

Step 1. For all $m$, $\delta$ and $C$, $tt=0$ in a \emph{tabu list}.\\ 
Step 2. Find the first integer solution for $P$ by Xpress and set it as \emph{current solution}.\\
Step 3. Create a \emph{candidate list} of size \emph{SizeList}:\\
\hspace*{1 cm} 1. Choose randomly $rm$, $rd$ and $rC$ from $m$'s, $\delta$'s and $C$'s in \emph{current solution}.\\
\hspace*{1 cm} 2. If $tt=0$, unfix these variables.\\
\hspace*{1 cm} 3. Solve the problem using branch and bound (Xpress).\\
\hspace*{1 cm} 4. Repeat \emph{SizeList} times.\\
Step 4. Sort the $candidate$ $list$ in decreasing order of the objective function value.\\
Step 5. Apply procedure LSF, LSU or LSVNS with the best first in the candidate list as input.\\
Step 6. Set the \emph{current solution} as the best solution in the \emph{candidate list}.\\
Step 7. Update the \emph{tabu list} for each $m$, $\delta$ and $C$ on the \emph{current solution}:\\
\hspace*{1 cm} 1. If current $tt=0$ then $tt=maxtt$, else\\
\hspace*{1 cm} 2. $tt=tt-1$.\\
Step 8. Repeat Steps $3-7$ until a maximum number of iterations is reached.\\

\end{algorithm}

\newpage

\subsection{Computational Results} \label{crTS}

Because we did not have access to real case information we have generated some artificial data as described below. The intervals of random values are based on the small example of five arteries and seven signals provided by \cite{Little66} which was solved with branch and bound by hand. All random data take the same values for outbound and inbound direction. It must be noted that even small grid graphs are very dense.

\vspace{0.5cm}

Let $U(a,b)$ be a continuous uniform distribution on interval $(a,b)$,

\vspace{0.2cm}

\begin{itemize}
\item The lengths of the arcs of the grid follow a distribution $U(140,600)$ (meters). 
\item Red times $r$ follow a distribution $U(0.4,0.6)$ (periods). 
\item Times to turn left $\ell$ follow a distribution  $U(0.25r,0.38r)$ (seconds). 
\item Min/max common period $T_{min}$/$T_{max}$, follows a distribution  $U(40,60)$/$U(90,110)$  (seconds).
\item Limits of velocities lower/upper $e$/$f$ follow a distribution $U(12,14)$/$U(15,16)$ (meters/second). 
\item Limits on changes in reciprocal speed lower/upper $1/h$/$1/g$ $=0.012$/$-0.012$ (meters/second)$^{-1}$.
\item All $\tau_{ai}$'s and $\overline{\tau}_{ai}$'s were set to 0.
\item All weights on the objective function were set to 1.
\end{itemize}

\vspace{0.3cm}

We used a PC Intel(R) Xeon(R) 3.40GHz $16.0$ (RAM). Strictly fundamental cycle basis \citep{LiebC,LandR} for each graph $G_{r \times c}$ were found with Mathematica version $10.1$. The algorithms were coded with Xpress Mosel version $3.4.2$ and the solver used was Xpress Optimizer version $24.01.04$.

\vspace{0.2cm}

Table \ref{tableFINALsmall} shows the results for several small grid graph instances generated randomly considering all parameters and variables in model \ref{maxbandn} including the bounds (\ref{limitsm}) and (\ref{limitsC}).

\begin{table}[ht!]
  \centering
  \caption{Computational Results for TSILP procedure (small instances).}
  \label{tableFINALsmall}%
    \footnotesize
\renewcommand{\arraystretch}{1.5} 
    \begin{tabular}{r@{\extracolsep{0.3cm}}r@{\extracolsep{0.3cm}}c@{\extracolsep{0.1cm}}r@{\extracolsep{0.3cm}}r@{\extracolsep{0.1cm}}r@{\extracolsep{0.1cm}}r@{\extracolsep{0.3cm}}r@{\extracolsep{0.1cm}}r@{\extracolsep{0.1cm}}r@{\extracolsep{0.1cm}}r@{\extracolsep{0.4cm}}r@{\extracolsep{0.1cm}}r@{\extracolsep{0.1cm}}r@{\extracolsep{0.1cm}}r@{\extracolsep{0.4cm}}r@{\extracolsep{0.1cm}}r@{\extracolsep{0.1cm}}r@{\extracolsep{0.1cm}}r@{\extracolsep{0.4cm}}r@{\extracolsep{0.1cm}}r@{\extracolsep{0.1cm}}r@{\extracolsep{0.1cm}}r}
    \toprule
         &    &  \multicolumn{2}{c}{Exact}    & \multicolumn{3}{c}{Global} & \multicolumn{4}{c}{Global/LS} & \multicolumn{4}{c}{TSILP-LSF} & \multicolumn{4}{c}{TSILP-LSU} & \multicolumn{4}{c}{TSILP-LSVNS}\\
    \cmidrule(){3-4}
    \cmidrule(){5-7}
    \cmidrule(){8-11}
    \cmidrule(){12-15}
    \cmidrule(){16-19}
    \cmidrule(){20-23} 

size & \# & OF* & t & iter & sl & tt & iLS & rm & rd & rC & avg & worst & best & avgt & avg & worst & best & avgt & avg & worst & best & avgt\\
    \midrule
    3x3     & 1   & 3.22 & 0 & 10    & 5     & 3     & 5     & 2     & 2     & 2     & 2.73  & 2.71  & 2.91  & 6     & 3.01  & 3.01  & \underline{3.02}  & 7     & 3.00  & 2.83  & \underline{3.02}  & 9 \\
          &       &       &       & 10    & 5     & 3     & 10    & 2     & 2     & 2     & 2.72  & 2.71  & 2.81  & 8     & 3.01  & 3.01  & \underline{3.02}  & 10    & 3.02  & 3.01  & \underline{3.02}  & 15 \\
          &       &       &       & 30    & 10    & 3     & 10    & 4     & 4     & 4     & 2.71  & 2.71  & 2.76  & 33    & 3.02  & \underline{3.02}  & \underline{3.02}  & 45    & 3.02  & \underline{3.02}  & \underline{3.02}  & 59 \\
          &       &       &       & 50    & 10    & 3     & 20    & 4     & 4     & 4     & 2.71  & 2.71  & 2.71  & 73    & 3.02  & \underline{3.02}  & \underline{3.02}  & 123   & 3.02  & \underline{3.02}  & \underline{3.02}  & 174 \\
\cmidrule{5-23}         & 2   & 3.80 & 0 & 10    & 5     & 3     & 5     & 2     & 2     & 2     & 2.61  & 1.24  & 3.76  & 5     & 3.48  & 1.38  & \underline{3.80}  & 6     & 3.79  & 3.67  & \underline{3.80}  & 8 \\
          &       &       &       & 10    & 5     & 3     & 10    & 2     & 2     & 2     & 3.24  & 1.24  & 3.77  & 10    & 3.78  & 3.72  & \underline{3.80}  & 9     & 3.80  & 3.76  & \underline{3.80}  & 13 \\
          &       &       &       & 30    & 10    & 3     & 10    & 4     & 4     & 4     & 3.69  & 3.69  & 3.69  & 31    & 3.69  & 3.69  & 3.69  & 39    & 3.69  & 3.69  & 3.69  & 52 \\
          &       &       &       & 50    & 10    & 3     & 20    & 4     & 4     & 4     & 3.69  & 3.69  & 3.69  & 72    & 3.69  & 3.69  & 3.69  & 102   & 3.69  & 3.69  & 3.69  & 145 \\
    \midrule
    5x5     & 3   & 4.77 & 7 & 10    & 5     & 3     & 5     & 2     & 2     & 2     & 3.57  & 2.63  & 4.16  & 11    & 4.05  & 3.43  & 4.31  & 14    & 4.11  & 3.73  & 4.29  & 17 \\
          &       &       &       & 10    & 5     & 3     & 10    & 2     & 2     & 2     & 3.86  & 3.29  & 4.06  & 14    & 4.08  & 3.52  & 4.30  & 17    & 4.17  & 3.83  & 4.35  & 19 \\
          &       &       &       & 30    & 10    & 3     & 10    & 4     & 4     & 4     & 4.16  & 4.01  & 4.41  & 55    & 4.31  & 4.18  & \underline{4.73}  & 96    & 4.36  & 4.18  & \underline{4.73}  & 117 \\
          &       &       &       & 50    & 10    & 3     & 20    & 4     & 4     & 4     & 4.14  & 3.94  & 4.18  & 145   & 4.33  & 3.95  & 4.71  & 257   & 4.42  & 4.23  & 4.72  & 332 \\
\cmidrule{5-23}         & 4   & 5.20 & 6 & 10    & 5     & 3     & 5     & 2     & 2     & 2     & 3.88  & 3.75  & 3.98  & 7     & 4.09  & 3.84  & 4.43  & 11    & 4.24  & 3.98  & 4.49  & 11 \\
          &       &       &       & 10    & 5     & 3     & 10    & 2     & 2     & 2     & 3.91  & 3.74  & 4.11  & 14    & 4.25  & 3.78  & 4.77  & 22    & 4.21  & 4.11  & 4.26  & 29 \\
          &       &       &       & 30    & 10    & 3     & 10    & 4     & 4     & 4     & 4.47  & 3.97  & 4.61  & 58    & 4.77  & 4.66  & \underline{4.86}  & 84    & 4.74  & 4.46  & \underline{4.86}  & 120 \\
          &       &       &       & 50    & 10    & 3     & 20    & 4     & 4     & 4     & 4.45  & 3.97  & 4.61  & 145   & 4.78  & 4.66  & \underline{4.86}  & 260   & 4.83  & 4.77  & \underline{4.86}  & 326 \\
    \midrule
    6x6     & 5   & 5.18 & 116 & 10    & 5     & 3     & 5     & 2     & 2     & 2     & 3.11  & 2.01  & 3.52  & 15    & 3.51  & 3.28  & 3.94  & 17    & 3.54  & 3.18  & 3.97  & 20 \\
          &       &       &       & 10    & 5     & 3     & 10    & 2     & 2     & 2     & 3.12  & 2.49  & 3.48  & 21    & 3.49  & 3.07  & 3.79  & 33    & 3.48  & 3.07  & 3.86  & 39 \\
          &       &       &       & 30    & 10    & 3     & 10    & 4     & 4     & 4     & 3.58  & 3.38  & 3.73  & 85    & 4.07  & 3.57  & 4.24  & 128   & 4.16  & 3.91  & 4.24  & 152 \\
          &       &       &       & 50    & 10    & 3     & 20    & 4     & 4     & 4     & 3.73  & 3.49  & 4.18  & 190   & 4.23  & 3.93  & \underline{4.46}  & 345   & 4.31  & 4.12  & \underline{4.46}  & 420 \\
\cmidrule{5-23}         & 6   & 4.74 & 1270 & 10    & 5     & 3     & 5     & 2     & 2     & 2     & 3.35  & 1.63  & 4.16  & 12    & 4.18  & 3.92  & 4.31  & 17    & 4.16  & 3.92  & 4.31  & 19 \\
          &       &       &       & 10    & 5     & 3     & 10    & 2     & 2     & 2     & 3.57  & 1.63  & 4.15  & 17    & 4.19  & 3.74  & 4.35  & 25    & 4.26  & 4.10  & 4.35  & 36 \\
          &       &       &       & 30    & 10    & 3     & 10    & 4     & 4     & 4     & 4.24  & 4.17  & 4.33  & 92    & 4.32  & 4.14  & \underline{4.37}  & 128   & 4.36  & 4.35  & \underline{4.37}  & 152 \\
          &       &       &       & 50    & 10    & 3     & 20    & 4     & 4     & 4     & 4.28  & 4.22  & 4.33  & 193   & 4.36  & 4.35  & \underline{4.37}  & 274   & 4.36  & 4.35  & \underline{4.37}  & 295 \\
    \bottomrule
    \end{tabular}%
\end{table}%

\newpage

The meanings of the headers are as follows:

\begin{itemize}
      \item size: size of the problem.
      \item \#: instance number.
      \item Exact (Xpress branch-and-bound).
      \begin{itemize}
      \item OF*: optimal value of the objective function. 
      \item t: running time (seconds).
      \end{itemize}
	  \item Global (stands for whole procedure).
	  \begin{itemize}	        
      \item iter: number of tabu search iterations.
      \item sl: size list.
      \item tt: tenure time in the memory list (iterations).
      \end{itemize}      
	  \item Global/LS (stands for whole and local search procedures).
	  \begin{itemize}	        
      \item iLS: number of iterations of local search.
      \item rm, rd, rC: number of variables $m$'s, $\delta$'s and $C$'s chosen for the candidate list and the local search.
	  \end{itemize}
	  \item TSILP-(LSF, LSU, LSVNS).	
	  \begin{itemize}      
      \item avg, worst, best: Average, worst and best case objective function value.
      \item avgt: average time (seconds).
      \end{itemize}
\end{itemize}

Each problem has been solved ten times with four different parameters settings. The same number of integer variables $rm$, $rd$ and $rC$ were used in both of them to create a candidate list and to run the local search procedure. The best objective function values obtained among the different heuristic algorithms are underlined.

For each instances the optimal objective function could be obtained by Xpress in short time. TSILP-LSU and TSILP-LSVNS could meet the optimal for problem \#2 and both algorithms performed almost the same in the most cases. See Figure \ref{boxplot5} as an example, red dots represents the mean OF values.

\begin{figure}[ht!]
  \centering
    \includegraphics[width=1\textwidth]{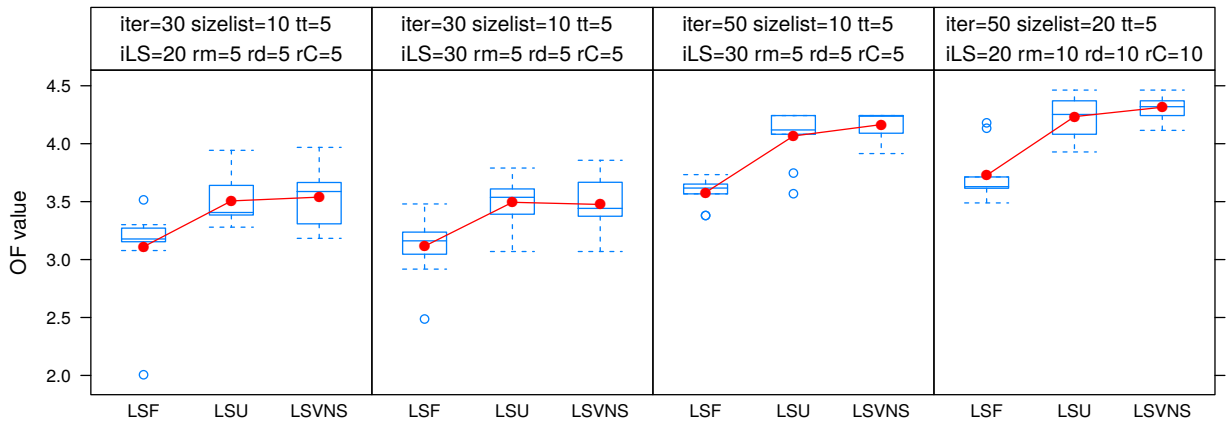}
  \caption{Box-Plot of 10 runs for experiments TSILP-LS on $G_{6 \times 6}$ (instance 5).}
  \label{boxplot5}
\end{figure}

For $5 \times 5 $, $6 \times 6 $ sizes problems the increase in the number of iterations, size list and selected random variables produced better results. It is clear that as the problem size increases, the heuristics times are more competitive.

Results for larger instances are shown in Table \ref{tableFINALlarge}. Xpress was not able to find the optimal solution for any of them within a time limit of 3 hours. For these instances it can be observed that TSILP-LSVNS performs better. See Figures \ref{boxplot10} and \ref{boxplot14}.

\begin{table}[ht!]
  \centering
  \caption{Computational Results for TSILP procedure (large instances).}
  \label{tableFINALlarge}%
    \footnotesize
\renewcommand{\arraystretch}{1.5} 
    \begin{tabular}{r@{\extracolsep{0.2cm}}r@{\extracolsep{0.3cm}}r@{\extracolsep{0.1cm}}r@{\extracolsep{0.1cm}}r@{\extracolsep{0.3cm}}r@{\extracolsep{0.1cm}}r@{\extracolsep{0.1cm}}r@{\extracolsep{0.1cm}}r@{\extracolsep{0.35cm}}r@{\extracolsep{0.2cm}}r@{\extracolsep{0.2cm}}r@{\extracolsep{0.2cm}}r@{\extracolsep{0.35cm}}r@{\extracolsep{0.2cm}}r@{\extracolsep{0.2cm}}r@{\extracolsep{0.2cm}}r@{\extracolsep{0.35cm}}r@{\extracolsep{0.2cm}}r@{\extracolsep{0.2cm}}r@{\extracolsep{0.2cm}}r}
    \toprule
         &        & \multicolumn{3}{c}{Global} & \multicolumn{4}{c}{Global/LS} & \multicolumn{4}{c}{TSILP-LSF} & \multicolumn{4}{c}{TSILP-LSU} & \multicolumn{4}{c}{TSILP-LSVNS}\\
    \cmidrule(){3-5}
    \cmidrule(){6-9}
    \cmidrule(){10-13}
    \cmidrule(){14-17}
    \cmidrule(){18-21} 

size & \#  & iter & sl & tt & iLS & rm & rd & rC & avg & worst & best & avgt & avg & worst & best & avgt & avg & worst & best & avgt\\
    \midrule
    7x7   & 7 & 30    & 10    & 5     & 20    & 5     & 5     & 5     & 4.33  & 4.24  & 4.34  & 292   & 4.51  & 4.35  & 4.66  & 484   & 4.63  & 4.56  & 4.68  & 546 \\
          &       & 30    & 10    & 5     & 30    & 5     & 5     & 5     & 4.34  & 4.27  & 4.46  & 213   & 4.51  & 4.35  & \underline{4.72}  & 440   & 4.58  & 4.36  & \underline{4.72}  & 520 \\
          &       & 50    & 10    & 5     & 30    & 5     & 5     & 5     & 4.34  & 4.34  & 4.35  & 300   & 4.56  & 4.35  & 4.70  & 625   & 4.66  & 4.53  & \underline{4.72}  & 815 \\
          &       & 50    & 20    & 5     & 30    & 10    & 10    & 10    & 4.37  & 4.24  & 4.55  & 623   & 4.67  & 4.56  & \underline{4.72}  & 1469  & 4.70  & 4.66  & \underline{4.72}  & 1700 \\
\cmidrule{3-21}          & 8 & 30    & 10    & 5     & 20    & 5     & 5     & 5     & 3.11  & 2.88  & 3.19  & 175   & 3.28  & 3.19  & 3.46  & 347   & 3.40  & 3.19  & 3.55  & 417 \\
          &       & 30    & 10    & 5     & 30    & 5     & 5     & 5     & 3.15  & 2.88  & 3.19  & 162   & 3.31  & 3.12  & 3.48  & 329   & 3.41  & 3.19  & 3.83  & 425 \\
          &       & 50    & 10    & 5     & 30    & 5     & 5     & 5     & 3.21  & 3.19  & 3.38  & 399   & 3.35  & 3.19  & 3.55  & 613   & 3.63  & 3.23  & 3.99  & 817 \\
          &       & 50    & 20    & 5     & 30    & 10    & 10    & 10    & 3.29  & 3.19  & 3.82  & 581   & 3.52  & 3.22  & 4.30  & 1113  & 3.81  & 3.36  & \underline{4.39}  & 1323 \\
    \midrule
    8x8   & 9 & 30    & 10    & 5     & 20    & 5     & 5     & 5     & 4.08  & 3.85  & 4.28  & 219   & 4.22  & 4.20  & 4.28  & 398   & 4.23  & 4.20  & 4.28  & 483 \\
          &       & 30    & 10    & 5     & 30    & 5     & 5     & 5     & 3.99  & 3.85  & 4.19  & 281   & 4.21  & 4.20  & 4.24  & 715   & 4.24  & 4.20  & 4.28  & 910 \\
          &       & 50    & 10    & 5     & 30    & 5     & 5     & 5     & 4.04  & 3.85  & 4.19  & 472   & 4.23  & 4.20  & \underline{4.33}  & 806   & 4.25  & 4.20  & \underline{4.33}  & 900 \\
          &       & 50    & 20    & 5     & 30    & 10    & 10    & 10    & 4.22  & 4.19  & 4.28  & 778   & 4.28  & 4.20  & \underline{4.33}  & 1710  & 4.30  & 4.24  & \underline{4.33}  & 1760 \\
\cmidrule{3-21}          & 10 & 30    & 10    & 5     & 20    & 5     & 5     & 5     & 2.25  & 2.08  & 2.84  & 167   & 3.07  & 2.54  & 3.37  & 397   & 3.10  & 2.75  & 3.53  & 524 \\
          &       & 30    & 10    & 5     & 30    & 5     & 5     & 5     & 2.25  & 2.01  & 2.84  & 218   & 3.03  & 2.38  & 3.49  & 601   & 3.12  & 2.77  & 3.49  & 727 \\
          &       & 50    & 10    & 5     & 30    & 5     & 5     & 5     & 2.33  & 2.08  & 2.96  & 382   & 3.35  & 2.84  & 3.72  & 710   & 3.42  & 2.84  & 3.61  & 1186 \\
          &       & 50    & 20    & 5     & 30    & 10    & 10    & 10    & 3.01  & 2.68  & 3.22  & 567   & 3.54  & 3.27  & 3.92  & 1345  & 3.72  & 3.29  & \underline{4.10}  & 1688 \\
    \midrule
    9x9   & 11 & 30    & 10    & 5     & 20    & 5     & 5     & 5     & 3.33  & 3.00  & 3.70  & 282   & 4.15  & 3.61  & 4.35  & 524   & 4.15  & 3.61  & 4.31  & 620 \\
          &       & 30    & 10    & 5     & 30    & 5     & 5     & 5     & 3.40  & 3.00  & 4.02  & 324   & 4.04  & 3.61  & 4.31  & 620   & 4.20  & 3.61  & 4.49  & 755 \\
          &       & 50    & 10    & 5     & 30    & 5     & 5     & 5     & 3.43  & 3.00  & 3.65  & 522   & 4.25  & 3.61  & 4.88  & 893   & 4.93  & 4.75  & 5.01  & 928 \\
          &       & 50    & 20    & 5     & 30    & 10    & 10    & 10    & 4.14  & 3.79  & 4.42  & 479   & 4.36  & 4.26  & 4.49  & 1131  & 4.80  & 4.38  & \underline{5.20}  & 1206 \\
\cmidrule{3-21}          & 12 & 30    & 10    & 5     & 20    & 5     & 5     & 5     & 4.17  & 4.07  & 4.31  & 270   & 4.31  & 4.24  & 4.71  & 541   & 4.36  & 4.24  & 4.71  & 646 \\
          &       & 30    & 10    & 5     & 30    & 5     & 5     & 5     & 4.18  & 4.07  & 4.31  & 316   & 4.25  & 4.22  & 4.31  & 699   & 4.28  & 4.24  & 4.42  & 848 \\
          &       & 50    & 10    & 5     & 30    & 5     & 5     & 5     & 4.18  & 4.07  & 4.42  & 539   & 4.26  & 4.24  & 4.31  & 901   & 4.39  & 4.24  & \underline{4.74}  & 1223 \\
          &       & 50    & 20    & 5     & 30    & 10    & 10    & 10    & 4.27  & 4.12  & 4.42  & 658   & 4.43  & 4.24  & 4.56  & 1324  & 4.45  & 4.31  & 4.61  & 1698 \\
    \midrule
    10x10 & 13 & 30    & 10    & 5     & 20    & 5     & 5     & 5     & 1.63  & 1.55  & 1.91  & 9937  & 1.85  & 1.76  & 1.91  & 10123 & 1.86  & 1.76  & 1.91  & 10259 \\
          &       & 30    & 10    & 5     & 30    & 5     & 5     & 5     & 1.75  & 1.55  & 1.91  & 10007 & 1.84  & 1.57  & 1.91  & 10278 & 1.89  & 1.72  & 1.91  & 10488 \\
          &       & 50    & 10    & 5     & 30    & 5     & 5     & 5     & 1.71  & 1.55  & 1.91  & 10218 & 1.86  & 1.66  & 1.91  & 10714 & 1.91  & 1.90  & 1.91  & 11058 \\
          &       & 50    & 20    & 5     & 30    & 10    & 10    & 10    & 1.89  & 1.72  & 1.91  & 10553 & 1.91  & 1.91  & 1.91  & 11341 & 1.94  & 1.73  & \underline{2.01}  & 11729 \\
\cmidrule{3-21}          & 14 & 30    & 10    & 5     & 20    & 5     & 5     & 5     & 2.36  & 2.30  & 2.46  & 474   & 2.86  & 2.56  & 2.97  & 778   & 2.91  & 2.75  & 2.97  & 918 \\
          &       & 30    & 10    & 5     & 30    & 5     & 5     & 5     & 2.45  & 2.30  & 2.90  & 554   & 2.85  & 2.52  & 3.03  & 769   & 2.94  & 2.79  & 3.05  & 958 \\
          &       & 50    & 10    & 5     & 30    & 5     & 5     & 5     & 2.40  & 2.30  & 2.55  & 648   & 2.93  & 2.84  & 3.05  & 1407  & 2.95  & 2.79  & 3.05  & 2016 \\
          &       & 50    & 20    & 5     & 30    & 10    & 10    & 10    & 2.98  & 2.77  & 3.05  & 876   & 3.13  & 3.04  & 3.21  & 2082  & 3.20  & 3.05  & \underline{3.24}  & 3238 \\
    \bottomrule
    \end{tabular}%
\end{table}%

\begin{figure}[ht!]
  \centering
    \includegraphics[width=1\textwidth]{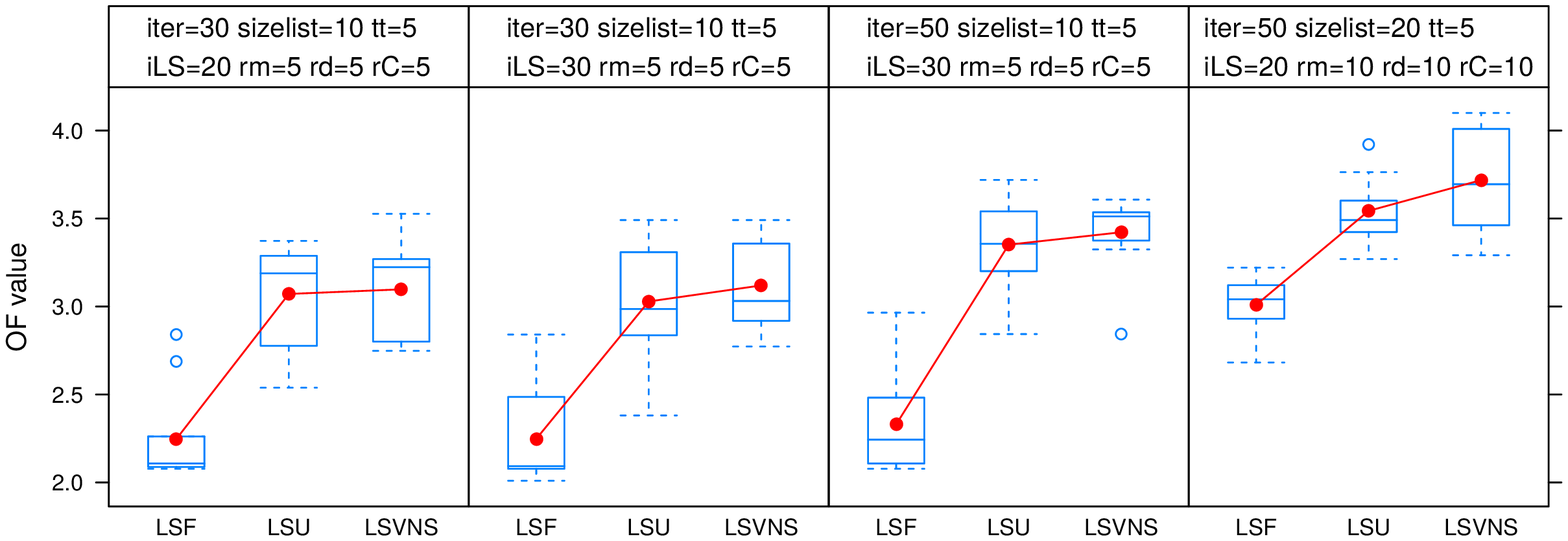}
  \caption{Box-Plot of 10 runs for experiments TSILP-LS on $G_{8 \times 8}$ (instance 10).}
  \label{boxplot10}
\end{figure}

\vspace{0.1cm}

\begin{figure}[ht!]
  \centering
    \includegraphics[width=1\textwidth]{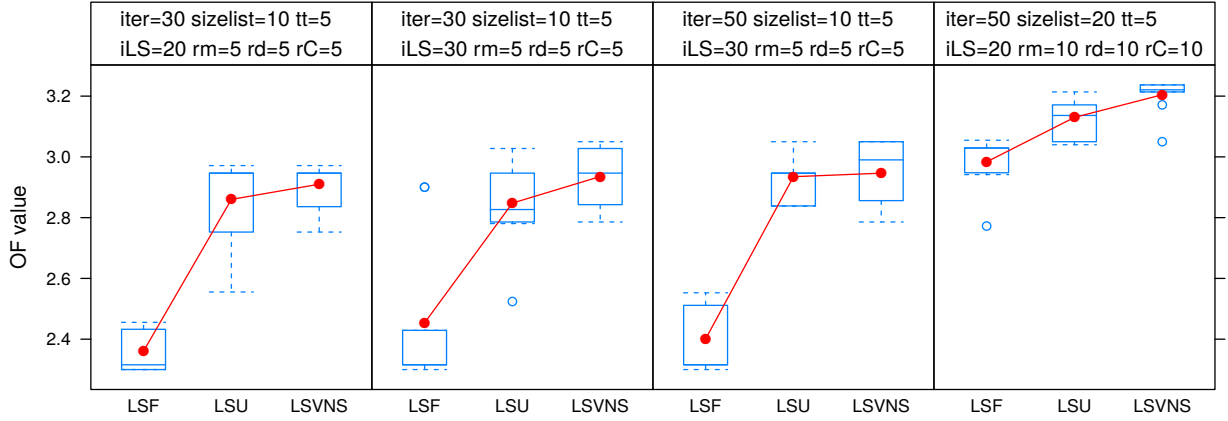}
  \caption{Box-Plot of 10 runs for experiments TSILP-LS on $G_{10 \times 10}$ (instance 14).}
  \label{boxplot14}
\end{figure}

The running times include the time needed to find the first initial feasible solution using the standard branch and bound algorithm in Xpress. For instance \#13 this time was 9657 seconds, but just 152 for instance \#14. The remainder took less than 11 seconds.

In most of the cases the average solution of TSILP-LSVNS improved the average solutions obtained with the first two methods. TSILP-LSF took less time, but almost never found the best objective function value obtained by the other algorithms.

Due to the vast amount of initial data, having the integer model ready to use before applying any algorithm takes some time that was not considered in the running times, but for the largest instances tested it averaged $20.73$ seconds. This is time used by Xpress to generate the model that will be solved.

The best results were found for the largest instances. In this case the three algorithms, even in the worst cases, reached much better solution than Xpress. Table \ref{10best} shows the first integer solution (f\_OF), the time to meet that solution (t\_fs) and the best integer solution after wait a considerable time in relation to those reached in the experiments with the proposed algorithms, 3 hours of running time (t$>$3h), by Xpress.

\vspace{0.3cm}

\begin{table}[ht!]
  \centering
  \caption{10$\times$10 instances vs Xpress.}
  \label{10best}
      \footnotesize
    \begin{tabular}{c@{\extracolsep{0.3cm}}c@{\extracolsep{0.4cm}}r@{\extracolsep{0.5cm}}r@{\extracolsep{0.5cm}}r@{\extracolsep{0.4cm}}r@{\extracolsep{0.3cm}}r@{\extracolsep{0.4cm}}r@{\extracolsep{0.3cm}}r@{\extracolsep{0.4cm}}r@{\extracolsep{0.3cm}}r}
    \toprule
          &       & \multicolumn{3}{c}{Exact} & \multicolumn{2}{c}{TSILP\_LSF} & \multicolumn{2}{c}{TSILP\_LSU} & \multicolumn{2}{c}{TSILP\_LSVNS}\\
    \cmidrule(){3-5}
    \cmidrule(){6-7}
    \cmidrule(){8-9}
    \cmidrule(){10-11}
size & \#  & f\_OF & t\_fs & t$>$3h & worst & best & worst & best & worst & best \\
    \midrule
    10x10 & 13    & 0.14  & 9657  & 1.55  & 1.55  & 1.91  & 1.57  & 1.91  & 1.72  & 2.01 \\
          & 14    & 0.54  & 152   & 2.19  & 2.30  & 3.05  & 2.52  & 3.21  & 2.75  & 3.24 \\
    \bottomrule
    \end{tabular}%
\end{table}%

\subsection{Conclusions}\label{conclusionsTLSP}

We started with a complete review of MAXBAND. We generalized the integer variables bounds given in \cite{Little66} to the network case. Cycle integer variables bounds are given as well.

We proposed a heuristic algorithm based on tabu search that takes advantage of the mixed integer linear MAXBAND model to obtain feasible solutions. During the search for a solution, we solve problems in a reduced feasible space, as the proposed algorithm begins with a feasible solution and then we use the obtained integer values to be used in subsequent iterations. The results were better when a serial application of LSF and LSU was done in a local search (VNS). As seen in the experiments, the best results were obtained in the largest instances tested.

In practice, there are many other aspects that require further attention and that we did not take into account in our experiments, such as prioritizing certain arteries with high traffic. This would change the weights in the objective function.

\section{Acknowledgments}
The research of Xavier Cabezas has been supported in part by SENESCYT-Ecuador (National Secretary of Higher Education,
Science, Technology and Innovation of Ecuador).\\

\noindent The research of Sergio Garc\'ia has been supported by Fundaci\'on S\'eneca (project 19320/PI/14, Regi\'on of Murcia, Spain).


\bibliography{bibliographyTLTS}
\bibliographystyle{apalike}

\end{document}